\documentclass[a4, 11pt]{amsart}

\usepackage{amsmath}
\usepackage{amssymb}
\usepackage{amsthm}

\newcommand{\F}{\mathcal{F}}
\newcommand{\R}{\mathbb R}
\newcommand{\T}{\mathbb T}
\newcommand{\N}{\mathbb N}
\newcommand{\Z}{\mathbb Z}

\newcommand{\supp}{\mbox{supp}}

\newcommand{\e}{\varepsilon}
\newcommand{\lb}{\langle}
\newcommand{\rb}{\rangle}
\newcommand{\la}{\langle}
\newcommand{\ra}{\rangle}
\newcommand{\ls}{\lesssim}
\newcommand{\gs}{\gtrsim}
\newtheorem{lemma}{Lemma} 
\newtheorem{kor}{Corollary} 

\newtheorem{prop}{Proposition}
\newtheorem{theorem}{Theorem}

\newcommand{\pd}{\partial}

\title[KP-II type equations]{On KP-II type equations on cylinders}

\author[A.~Gr{\"u}nrock]{Axel~Gr{\"u}nrock} 
\author[M.~Panthee]{Mahendra~Panthee}
\author[J.~Drumond~Silva]{Jorge~Drumond~Silva}

\address{Axel~Gr{\"u}nrock: Rheinische Friedrich-Wilhelms-Universit\"at Bonn,
Mathematisches Institut, Beringstra{\ss}e 1, 53115 Bonn, Germany.}
\email{gruenroc@math.uni-bonn.de}
\address{Mahendra~Panthee:
Centro de Matem\'atica, Universidade do Minho, Campus de Gualtar, 4710-057 Braga, Portugal.}
\email{mpanthee@math.uminho.pt}
\address{Jorge~Drumond~Silva:
Center for Mathematical Analysis, Geometry and Dynamical Systems,
Departamento de Matem\'atica,
Instituto Superior T\'ecnico, 
 Av. Rovisco Pais, 1049-001 Lisboa, Portugal.}
\email{jsilva@math.ist.utl.pt}

\thanks{A. Gr\"unrock was partially supported by the Deutsche Forschungsgemeinschaft, Sonderforschungsbereich 611.
M. Panthee was partially supported through the program
POCI 2010/FEDER. J. Drumond Silva was partially supported through the program
POCI 2010/FEDER and by the project POCI/FEDER/MAT/55745/2004.}


\begin{document}

\maketitle

\begin{abstract}
In this article we study the generalized dispersion version of the Kadomtsev-Petviashvili
II equation, on $\T \times \R$ and $\T \times \R^2$. We start by
proving bilinear Strichartz type estimates, dependent only on the dimension of the domain but not on the
dispersion. Their analogues in terms of Bourgain spaces are then used as the main tool for the proof of bilinear estimates
of the nonlinear terms of the equation and consequently of local well-posedness 
for the Cauchy problem.
\end{abstract}

\tableofcontents

\section{Introduction}

In this paper, we consider the initial value problem (IVP) 
for generalized dispersion versions of the
 Kadomtsev-Petviashvili-II (defocusing) equation on $\T_x \times \R_y$
\begin{equation}
\label{KPIId2}
\left\{
\begin{array}{l}
\pd_t u - |D_x|^{\alpha}\pd_x u+\pd_x^{-1}\pd_y^2 u +u\pd_x u=0 \qquad u:\R_t \times \T_x \times \R_y \to \R,\\
u(0,x,y)=u_0(x,y),
\end{array}
\right.
\end{equation}
and on $\T_x \times \R^2_{y}$
\begin{equation}
\label{KPIId3}
\left\{
\begin{array}{l}
\pd_t u - |D_x|^{\alpha}\pd_x u+\pd_x^{-1}\Delta_{y} u +u\pd_x u=0 \qquad u:\R_t \times \T_x \times \R^2_{y} \to \R,\\
u(0,x,y)=u_0(x,y).
\end{array}
\right.
\end{equation}
We consider the dispersion parameter $\alpha \geq 2$. 
The operators $|D_x|^{\alpha}\pd_x$ and $\pd_x^{-1}$ are defined by their Fourier multipliers
$i|k|^{\alpha}k$ and $(ik)^{-1}$, respectively.

The classical Kadomtsev-Petviashvili (KP-I and KP-II) equations, when $\alpha=2$,
$$\pd_t u +\pd_x^3 u \pm \pd_x^{-1}\pd_y^2 u +u\pd_x u=0$$ 
are the natural two-dimensional generalizations of the 
Korteweg-de Vries (KdV) equation. They occur as models for the propagation of essentially one-dimensional
weakly nonlinear dispersive waves, with weak transverse effects. The focusing KP-I equation corresponds to the minus (-)
sign in the previous equation, whereas the defocusing KP-II is the one with the plus (+) sign.

The well-posedness of the Cauchy problem for the KP-II equation has been extensively studied, in recent years.
J. Bourgain \cite{B93} made a major breakthrough in the field by introducing Fourier restriction norm spaces, enabling 
a better control of the norms in the Picard iteration method applied to Duhamel's formula, and achieving
a proof of local well-posedness in $L^2(\T^2)$ (and consequently also global well-posedness, due to the conservation
of the $L^2$ norm in time). 
Since then, a combination of Strichartz estimates and specific techniques in the framework of Bourgain spaces
has been used by several authors to study KP-II type equations in several settings (see \cite{ILM}, \cite {IM01}, \cite{S93}, \cite{ST99},
\cite{ST00}, \cite{ST01}, \cite{T&T01}  and references therein). Recently, an optimal result was obtained
by M. Hadac \cite{H} for the generalized dispersion KP-II equation on $\R^2$, 
in which local well-posedness for the range of dispersions $\frac43 < \alpha \le 6$ was established
for the anisotropic Sobolev spaces $H^{s_1,s_2}(\R^2)$, provided
$s_1> \max{(1- \frac34 \alpha,\frac14 - \frac38 \alpha)}$, $s_2 \ge 0$, thus reaching the scaling
critical indices for $\frac43 < \alpha \le 2$. This includes the particular case $\alpha = 2$ corresponding to the classical KP-II equation.
 In this case the analysis was pushed further to the critical regularity by M. Hadac,
S. Herr, and H. Koch in \cite{HHK}, where a new type of basic function spaces - the so called $U^p$-spaces
introduced by H. Koch and D. Tataru - was used. Concerning the generalized dispersion KP-II equation on $\R^3$,
a general result was also shown by M. Hadac in \cite{HT}, which is optimal in the 
range $2 \le \alpha \le \frac{30}7$ by scaling considerations. For the particular case $\alpha = 2$, he obtained local well-posedness in
$H^{s_1,s_2}(\R^3)$ for $s_1 > \frac12$ and $s_2 > 0$.

In this article,  we aim to study the local well-posedness of the initial value problem for the 
general dispersion KP-II type equations \eqref{KPIId2} and \eqref{KPIId3}, on the cylinders
$\T_x \times \R_y$ and $\T_x \times \R^2_{y}$ respectively. We will show that the initial value
problem \eqref{KPIId2} is locally well-posed for data $u_0 \in H^{s_1,s_2}(\T \times \R)$ satisfying
the mean zero condition $\int_0^{2\pi}u(x,y) dx =0$, provided $ \alpha \ge 2$, $s_1 > \max{(\frac34 - \frac{\alpha}{2}, \frac18 - \frac{\alpha}{4})}$,
and $s_2 \ge 0$. Combined with the conservation of the $L^2_{xy}$-norm this local result implies
global (in time) well-posedness, whenever  $s_{1} \ge 0$ and $s_2=0$. Concerning \eqref{KPIId3} we will obtain
local well-posedness for $u_0 \in H^{s_1,s_2}(\T \times \R^2)$, satisfying again the mean zero
condition, in the following cases:
\begin{itemize}
\item $\alpha=2$, \hspace{1,4cm} $s_1 \ge \frac12$, \hspace{1,0cm} $s_2 >0$,
\item $2 < \alpha \le 5$, \qquad $s_1 > \frac{3-\alpha}2$, \qquad $s_2 \ge 0$,
\item $5 < \alpha$, \hspace{1,4cm} $s_1 >\frac{1 - \alpha}{4}$, \qquad $s_2 \ge 0$.
\end{itemize}
For $ \alpha > 3$ our result here is in, and below, $L^2_{xy}$. In this case we again obtain
global well-posedness, whenever $s_{1} \ge 0$ and $s_2=0$.

We proceed in three steps. First, in Section 2, we will establish bilinear Strichartz estimates for the linear versions of 
\eqref{KPIId2} and \eqref{KPIId3}, depending only
on the domain dimension but not on the dispersion parameter. We believe, these estimates are of interest on their own,
independently of their application here\footnote{For example our two-dimensional space time estimate, which is equally valid for the linearized KP-I equation, together
with the counterexamples presented later on gives a definite answer to a question raised by J. C. Saut and
N. Tzvetkov in \cite[remark on top of p. 460]{ST01}.}. In the second step, in Section 3, we will use these
Strichartz estimates to prove bilinear estimates for the nonlinear term of the equations, in Bourgain's Fourier restriction norm spaces. 
Finally, in Section 4, a precise statement will be given of our local well-posedness results for the associated initial value problems, 
with data in Sobolev 
spaces of low regularity. Their proofs follow a standard fixed point Picard iteration method applied to Duhamel's formula, 
using the bilinear estimates obtained in the previous section. In the appendix we provide a counterexample, due to H. Takaoka and N. Tzvetkov \cite{T&T}, 
concerning the two-dimensional case. This example shows the necessity of the lower bound $s_1 \ge \frac34 - \frac{\alpha}2$ 
and hence the optimality (except for the endpoint) of our two-dimensional result in the range $2 \le \alpha \le \frac52$.
For higher dispersion ($\alpha > \frac52$) we unfortunately lose optimality as a consequence of the case when an interaction of two high frequency factors
produces a very low resulting frequency. The same problem occurs in three space dimensions, but the effect is much weaker.
Here, by scaling considerations, our result is optimal for $2 \le \alpha \le 5$, and we leave the line of optimality only for
very high dispersion, when $\alpha > 5$.

\section{Strichartz Estimates}

Strichartz estimates have, in recent years, 
been playing a fundamental role in the proofs of local well-posedness results for the KP-II equation.
Their use has been a crucial ingredient for establishing the bilinear estimates associated to the 
nonlinear terms of the equations, in the Fourier restriction spaces developed by J. Bourgain, 
the proof of which 
is the central issue
in the Picard iteration argument in these spaces.
Bourgain \cite{B93} proved an $L^4-L^2$ Strichartz-type estimate, localized in frequency space, as the main tool for
obtaining the local well-posedness of the Cauchy problem 
in $L^2$, in the fully periodic two-dimensional case, $(x,y) \in \T^2$. J.C. Saut and N. Tzvetkov \cite{ST00} 
proceeded similarly, for the fifth order KP-II equation, also in $\T^2$ as well as $\T^3$.
Strichartz estimates for the fully nonperiodic versions of the (linearized) KP-II equations 
have also been extensively studied and used, both in the two
and in the three-dimensional cases. In these continuous domains, $\R^2$ and $\R^3$, the results follow typically 
by establishing time decay estimates for the spatial $L^{\infty}$ norms of the solutions, which in turn are usually 
obtained from the analysis of their oscillatory integral representations,
 as in \cite{BAS99},\cite{KPV91} or \cite{S93}. The Strichartz estimates obtained this way also exhibit 
a certain level of global smoothing effect for the solutions, which
naturally depends on the dispersion factor present in the equation.

As for our case, we prove bilinear versions of Strichartz type inequalities for the generalized KP-II equations 
on the cylinders $\T \times \R$ and $\T \times \R^2$. The main idea behind the proofs 
that we present below is to use the Fourier
transform $\F_x$ in the periodic $x$ variable only. And then, for the remaining $y$ variables, to apply the well 
known Strichartz inequalities for the Schr\"odinger
equation in $\R$ or $\R^2$. This way, we obtain estimates with a small loss of derivatives, 
but independent of the dispersion parameter.

So, consider the linear equations corresponding to \eqref{KPIId2} and \eqref{KPIId3},
\begin{equation}
\label{lin1}
\pd_t u - |D_x|^{\alpha}\pd_x u+\pd_x^{-1}\pd_y^2 u =0,
\end{equation}
respectively 
\begin{equation}
\label{lin2}
\pd_t u - |D_x|^{\alpha}\pd_x u+\pd_x^{-1}\Delta_{y} u =0.
\end{equation}
The phase function for both of these two equations 
is given by 
$$
\phi(\xi)=\phi_0(k)-\frac{|\eta|^2}{k},
$$ 
where $\phi_0(k)=|k|^{\alpha}k$ is the dispersion term and 
$\xi= (k,\eta)\in \Z^* \times \R$, respectively $\xi= (k,\eta)\in \Z^* \times \R^2$,
is the dual variable to $(x,y) \in \T \times \R$, respectively $(x,y) \in \T \times \R^2$,
so that the unitary evolution group for these linear equations is
$e^{it\phi(D)}$, where $D=-i\nabla$. For the initial data functions $u_0$, $v_0$ that we will consider below it is assumed 
that $\widehat{u_0}(0,\eta)=\widehat{v_0}(0,\eta)=0$
(mean zero condition).

The two central results of this section are the following.

\bigskip

\begin{theorem} \label{firstStr}
Let $\psi \in C_0^{\infty}(\R)$
be a \emph{time} cutoff function with $\psi\left|_{[-1,1]}\right.=1$ and $\supp{(\psi)} \subset (-2,2)$, and
let $u_0, v_0 : \T_x \times \R_{y} \to \R$ satisfy the mean zero condition in the $x$ variable. Then, 
for $s_{1,2} \ge 0$ such that $s_1+s_2=\frac 1 4$,
the following inequality holds:
\begin{equation}\label{str2}
\|\psi \; e^{it\phi(D)}u_0\:e^{it\phi(D)}v_0\|_{L^2_{txy}} \ls \|u_0\|_{H^{s_1}_x L^2_{y}}\|v_0\|_{H^{s_2}_x L^2_{y}}.
\end{equation}
\end{theorem} 

\bigskip

\begin{theorem} \label{secondStr}
Let $u_0, v_0 : \T_x \times \R_{y}^2 \to \R$ satisfy the mean zero condition in the $x$ variable.
Then, for $s_{1,2} \ge 0$ such that $s_1+s_2>1$, the following
inequality holds: 
\begin{equation}\label{str3}
\|e^{it\phi(D)}u_0\:e^{it\phi(D)}v_0\|_{L^2_{txy}} \ls \|u_0\|_{H^{s_1}_xL^2_{y}}\|v_0\|_{H^{s_2}_xL^2_{y}}.
\end{equation}
Choosing $u_0=v_0$ and $s_1=s_2=\frac 1 2 +$, we have in particular
\[\|e^{it\phi(D)}u_0\|_{L^4_{txy}}\ls\|u_0\|_{H^{\frac{1}{2}+}_xL^2_{y}}.\]
\end{theorem} 

\bigskip

Note that in the case of Theorem \ref{firstStr}, in the $\T_x \times \R_y$ domain, 
the Strichartz estimate is valid only locally in time. A proof of this fact 
is presented in the last result 
of this section

\begin{prop} \label{third}
There is no $s \in \R$ such that the estimate
$$
\| \; \big(e^{it\phi(D)}u_0\big)^2\|_{L^2_{t x y}} \ls \|D_x^s u_0\|_{L^2_{x y}}\|u_0\|_{L^2_{x y}},
$$
holds in general.
\end{prop}

The use of a cutoff 
function in time is therefore required in $\T \times \R$, whose presence will be fully exploited in the proof of Theorem 
\ref{firstStr}. In the case of Theorem \ref{secondStr}, where $y \in \R^2$, the result is valid globally in time 
and no such cutoff is needed to obtain the analogous 
Strichartz estimate\footnote{In any case, for our purposes of proving 
local well-posedness in time for the Cauchy problems \eqref{KPIId2} and
\eqref{KPIId3}, further on in this paper, this
issue of whether the Strichartz estimates are valid only locally or globally will not be relevant there.}.

As a matter of fact, in the three-dimensional case $\T \times \R^2$,  
the proof that we present is equally valid for the fully nonperiodic
three-dimensional domain, $\R^3$. As pointed out above, Strichartz estimates have been proved and used 
for the linear KP-II equation, in $\R^2$ and $\R^3$. But being usually derived through
oscillatory integral estimates and decay in time, 
they normally exhibit dependence on the particular dispersion under consideration, 
leading to different smoothing 
properties of the solutions. For
estimates independent of the dispersion term $\phi_0$ one can easily apply a 
dimensional analysis argument to determine - at least 
for homogeneous Sobolev spaces $\dot{H}^s$ - the indices that should be expected. So, for $\lambda \in \R$, 
if $u(t,x,y)$ is a solution to the linear equation \eqref{lin2} on $\R^3$, then $u^{\lambda}=C u(\lambda^3 t,
\lambda x, \lambda^2 y)$, $C\in\R$, is also a solution of the same equation, with initial data
$u^{\lambda}_0=C u_0(\lambda x, \lambda^2 y)$. An $L^4_{txy} - \dot{H}^s_xL^2_{y}$ estimate for this family
of scaled solutions then becomes
$$\lambda^{\frac 1 2 - s}\|u\|_{L^4_{txy}}\ls \|u_0\|_{\dot{H}^s_xL^2_{y}},$$
leading to the necessary condition $s= \frac 1 2$. Theorem \ref{secondStr}, for nonhomogeneous Sobolev spaces, 
touches this endpoint (not including it, though).

\subsection{Proof of the Strichartz estimate in the $\T \times \R$ case}

\begin{proof}[Proof of Theorem \ref{firstStr}] 
It is enough to prove the estimate \eqref{str2} when $s_1=1/4$ and $s_2=0$.

We have, for the space-time Fourier transform of the product of the two solutions to the linear equation
\footnote{Throughout the text
we will disregard
multiplicative constants, typically powers of $2\pi$, which are irrelevant for the final estimates.} 
\begin{equation}
\label{product}
\F(e^{it\phi(D)}u_0\:e^{it\phi(D)}v_0)(\tau, \xi)=
 \int_* \delta(\tau - \phi(\xi_1) - \phi(\xi_2))
 \widehat{u_0}(\xi_1)\widehat{v_0}(\xi_2)\mu(d\xi_1), 
\end{equation}
where $\int_* \mu(d\xi_1)= \sum_{\substack{k_1,k_2\ne 0 \\ k=k_1+k_2}} \int_{\eta_1+\eta_2=\eta} d \eta_1$, and
\[\phi(\xi_1) + \phi(\xi_2)=\phi_0(k_1) + \phi_0(k_2)-\frac{1}{k_1 k_2}(k\eta_1^2-2\eta k_1 \eta_1 + k_1 \eta^2).\]
Thus the argument of $\delta$, as a function of $\eta_1$, becomes
\[g(\eta_1):=\tau - \phi(\xi_1) - \phi(\xi_2)
 = \frac{1}{k_1 k_2}(k\eta_1^2-2\eta k_1 \eta_1 + k_1 \eta^2) +\tau - \phi_0(k_1) -\phi_0(k_2) .\]
The zeros of $g$ are
\[\eta_1^{\pm}=\frac{\eta k_1}{k}\pm \omega,\]
with
$$
\omega^2=\frac{k_1k_2}{k}\left(\phi_0(k_1) +\phi_0(k_2)-\frac{\eta^2}{k}-\tau\right),$$
whenever the right hand side is positive, and we have
\[|g'(\eta_1^{\pm})|=\frac{2|k|\omega}{|k_1k_2|}.\]
There are therefore two contributions $I^{\pm}$ to \eqref{product}, which are given by
\[I^{\pm}(\tau,\xi)=|k|^{-1}\sum_{\substack{k_1\\k_1,k_2\ne 0}}\frac{|k_1k_2|}{\omega}
\widehat{u_0}\left(k_1,\frac{\eta k_1}{k}\pm \omega\right)
\widehat{v_0}\left(k_2,\frac{\eta k_2}{k}\mp \omega\right),\]
and the space-time Fourier transform of $\psi\; e^{it\phi(D)}u_0\:e^{it\phi(D)}v_0$ then becomes
\begin{multline*}
\F(\psi\: e^{it\phi(D)}u_0\:e^{it\phi(D)}v_0)(\tau,\xi)=
\widehat{\psi} *_{\tau}\big(I^{+}(\tau,\xi)+I^{-}(\tau,\xi)\big)=\\
\int \widehat{\psi}(\tau-\tau_1)
\sum_{\substack{k_1\\k_1,k_2\ne 0}}\frac{|k_1k_2|}{\omega(\tau_1)|k|}\:\bigg[
\widehat{u_0}\left(k_1,\frac{\eta k_1}{k} + \omega(\tau_1)\right)
\widehat{v_0}\left(k_2,\frac{\eta k_2}{k} - \omega(\tau_1)\right)\\
\\+\widehat{u_0}\left(k_1,\frac{\eta k_1}{k} - \omega(\tau_1)\right)
\widehat{v_0}\left(k_2,\frac{\eta k_2}{k} + \omega(\tau_1)\right)\bigg]
d\tau_1.
\end{multline*}

For the $L^2$ estimate of this quantity we may assume, without loss of generality, that $k_1$ and $k_2$
are both positive (cf. pg. 460 in \cite{ST01}), so that $0<k_1,k_2 < k$.

We will now prove the result, by breaking up the sum into two cases which are estimated separately. 
\bigskip

\noindent \underline{Case I ($\omega(\tau_1)^2 > k_1k_2$)}:

In this case we start by using the elementary convolution estimate,
$$\|\widehat{\psi}*_{\tau}(I^{+}(\cdot,\xi)+I^{-}(\cdot,\xi))\left|_{\omega(\tau_1)^2 >  k_1k_2}\right.\|_{L^2_{\tau}}
\ls\|\widehat{\psi}\|_{L^1_{\tau}}\|(I^{+}(\cdot,\xi)+I^{-}(\cdot,\xi))\left|_{\omega^2 > k_1k_2}\right. \|_{L^2_{\tau}}.$$
Now, to estimate the $L^2$ norm of the sum, we fix any small $0<\epsilon<1/4$ and Cauchy-Schwarz gives
\begin{multline}
\label{CS}
|I^{\pm}(\tau,\xi)|_{|\omega(\tau)^2 > k_1k_2}\ls
\left(\sum_{k_1, k_2 > 0}k_1^{-2\e}\frac{k_1k_2}{k\:\omega(\tau)} \right)^{\frac{1}{2}}\\
\times \Bigg(\sum_{k_1, k_2 > 0}k_1^{2\e}\frac{k_1k_2}{k\: \omega(\tau)}
 \Big|\widehat{u_0}(k_1,\frac{\eta k_1}{k}\pm \omega(\tau))
 \widehat{v_0}(k_2,\frac{\eta k_2}{k}\mp \omega(\tau))\Big|^2\Bigg)^{\frac{1}{2}}.
\end{multline}
The condition $\omega(\tau)^2 > k_1k_2$ implies
$\frac{\omega(\tau)^2}{k_1k_2}=
|\frac 1 k (\phi_0(k_1)+\phi_0(k_2))-\frac{\eta^2}{k^2}-\frac{\tau}{k}|> 1$, 
so that 
$$\frac{\sqrt{k_1 k_2}}{\omega(\tau)}
\sim \frac{1}{\langle\frac 1 k (\phi_0(k_1)+\phi_0(k_2)) -\frac{\eta^2}{k^2}-\frac{\tau}{k}\rangle ^{\frac 1 2}}.$$
We also have $\frac{\sqrt{k_1 k_2}}{k}\leq 1$, from which we conclude then that the first factor in \eqref{CS} is  bounded 
by a constant $C_{\epsilon}$ independent of $k,\tau,\eta$, for
\begin{multline*}
\sum_{k_1, k_2 > 0}k_1^{-2\e}\frac{k_1k_2}{k \:\omega(\tau)} \ls 
\sum_{k_1, k_2 > 0}k_1^{-2\e}\frac{1}
{\langle \frac 1 k (\phi_0(k_1)+\phi_0(k_2)) -\frac{\eta^2}{k^2}-\frac{\tau}{k}\rangle ^{\frac 1 2}}\\
\ls \left( \sum_{k_1}k_1^{-2\e p}\right)^{\frac 1 p}\left( \sum_{k_1, k_2 > 0} \frac{1}
{\langle \frac 1 k (\phi_0(k_1)+\phi_0(k_2)) -\frac{\eta^2}{k^2}-\frac{\tau}{k}\rangle ^{\frac q 2}} \right)^{\frac 1 q}, 
\end{multline*}
which, using H\"older conjugate exponents $p>1/2\epsilon>2$ and $q=p/(p-1)<2$, as well as the easy calculus fact that
$$\sup_{\substack{a \in \R \\ k \in \N}}\:\sum_{k_1>0} 
\langle \frac 1 k (\phi_0(k_1)+\phi_0(k_2))-a\rangle ^{-\delta}\:\leq \:C_{\delta},$$
valid for any fixed $\alpha \geq 2$ and $\delta > 1/2$, implies
$$\sum_{k_1, k_2 > 0}k_1^{-2\e}\frac{k_1k_2}{k\:\omega(\tau)}\leq C_{\epsilon}.$$
We thus have
\begin{align*}
&\hspace{-1cm}\|I^{\pm}(\cdot,\xi)\left|_{\omega^2 > k_1k_2} \right. \|^2_{L^2_{\tau}} \\
 \ls &\sum_{k_1, k_2>0}k_1^{2\e}\int \frac{k_1k_2}{k\:\omega}
 \Big|\widehat{u_0}(k_1,\frac{\eta k_1}{k}\pm \omega)
 \widehat{v_0}(k_2,\frac{\eta k_2}{k}\mp \omega) \Big|^2 d \tau\\
 \ls & \sum_{k_1, k_2>0}k_1^{2\e}\int \Big|\widehat{u_0}(k_1,\frac{\eta k_1}{k}\pm \omega)
 \widehat{v_0}(k_2,\frac{\eta k_2}{k}\mp \omega) \Big|^2 d\omega.
\end{align*}
Here we have used $d \tau =\frac{2\omega k}{k_1k_2}d\omega$. Integrating with respect
to $d \eta$ and using the change of variables $\eta_+=\frac{\eta k_1}{k}\pm \omega$,
$\eta_-=\frac{\eta (k-k_1)}{k}\mp \omega$ with Jacobian $\mp1$ we arrive at
\[\|I^{\pm}(\cdot,\xi)\left|_{\omega^2 > k_1k_2} \right. \|^2_{L^2_{\tau}} \ls
\sum_{k_1, k_2>0}k_1^{2\e}\|\widehat{u_0}(k_1, \cdot)\|^2_{L^2_{\eta}}\|\widehat{v_0}(k_2, \cdot)\|^2_{L^2_{\eta}}.\]
Finally summing up over $k\ne 0$ we obtain
\[\|(I^{+}+I^{-})\left|_{\omega^2 > k_1k_2}\right. \|^2_{L^2_{\tau k \eta }} \ls
\|D_x^{\varepsilon}u_0\|^2_{L^2_{xy}}\|v_0\|^2_{L^2_{xy}}.\]

\bigskip

\noindent \underline{Case II ($\omega(\tau_1)^2\le k_1k_2$)}:

In this case $|\frac 1 k (\phi_0(k_1)+\phi_0(k_2))-\frac{\eta^2}{k^2}-\frac{\tau_1}{k}|\le 1$. Here we make the
further subdivision
$$1=\chi_{\{|\frac{\tau-\tau_1}{k}| \le 1\}}+\chi_{\{|\frac{\tau-\tau_1}{k}| > 1\}}.$$

When $|\frac{\tau-\tau_1}{k}| \le 1$ we have
$$\left|\frac 1 k (\phi_0(k_1)+\phi_0(k_2))-\frac{\eta^2}{k^2}-\frac{\tau}{k}\right|\le 
1 + \left|\frac{\tau-\tau_1}{k}\right| \le 2,$$
and for every fixed $\tau,k,\eta$ we have only a finite number of $k_1$'s
satisfying this condition. Therefore,
\begin{multline*}
\sum_{k_1>0} \int |\hat{\psi}(\tau-\tau_1)|\frac{k_1k_2}{k\:\omega(\tau_1)}
 \Big|\widehat{u_0}(k_1,\frac{\eta k_1}{k}\pm \omega(\tau_1))
 \widehat{v_0}(k_2,\frac{\eta k_2}{k}\mp \omega(\tau_1))\Big| 
\chi_{\{\omega(\tau_1)^2 \le k_1k_2\}}\chi_{\{|\frac{\tau-\tau_1}{k}| \le 1\}}d\tau_1  \\
\ls  \left(\sum_{k_1>0} \left(\int |\hat{\psi}(\tau-\tau_1)|\frac{k_1k_2}{k\:\omega(\tau_1)}
 \Big|\widehat{u_0}(k_1,\frac{\eta k_1}{k}\pm \omega(\tau_1))
 \widehat{v_0}(k_2,\frac{\eta k_2}{k}\mp \omega(\tau_1))\Big|d\tau_1 \right)^2 \right)^{\frac 1 2}.  
\end{multline*}

Now, the $L^2_{\tau \eta}$ norm of this quantity is bounded by
\begin{multline*}
\left(\sum_{k_1>0}\left\|\hat{\psi}*_{\tau_1}\Big(\frac{k_1k_2}{k\:\omega}
 \widehat{u_0}(k_1,\frac{\eta k_1}{k}\pm \omega)
 \widehat{v_0}(k_2,\frac{\eta k_2}{k}\mp \omega)\Big)\right\|_{L^2_{\tau \eta}}^2 \right)^{\frac 1 2}\\
=\left(\sum_{k_1>0}\left\|\psi \; e^{it(\phi_0(k_1)+\phi_0(k_2))}
e^{i \frac {t}{ k_1} \pd_y^2} \mathcal{F}_xu_0(k_1,\cdot) \;
e^{i \frac {t}{ k_2} \pd_y^2}\mathcal{F}_xv_0(k_2,\cdot) \right\|^2_{L^2_{ty}} \right)^{\frac 1 2},
\end{multline*}
where the equality is due to Plancherel's theorem, applied to the $t,y$ variables only.
By H\"older
\begin{multline}
\left\|\psi \; e^{it(\phi_0(k_1)+\phi_0(k_2))}
e^{i \frac {t}{ k_1} \pd_y^2} \mathcal{F}_xu_0(k_1,\cdot) \;
e^{i \frac {t}{ k_2} \pd_y^2}\mathcal{F}_xv_0(k_2,\cdot) \right\|_{L^2_{ty}} \\
\le \|\psi\|_{L^{4}_{t}}\|e^{i\frac{t}{k_1}\partial_y^2}\F_xu_0(k_1,\cdot)\|_{L^{4}_{t}L^{\infty}_{y}}
 \|e^{i\frac{t}{k_2}\partial_y^2}\F_xv_0(k_2,\cdot)\|_{L^{\infty}_{t}L^{2}_{y}}. \label{holder}
\end{multline}
The partial Fourier transform $\F_x$ of a free solution
with respect to the periodic $x$ variable only
\[\F_x \big( e^{it\phi(D)}u_0 \big)(k,y) = e^{it\phi_0 (k)}e^{i\frac{t}{k}\partial_y^2}\F_xu_0(k,y),\]
is, for every fixed $k$, a solution of the homogeneous linear Schr\"odinger
equation with respect to the nonperiodic $y$ variable and the rescaled time variable
$s:=\frac{t}{k}$, multiplied by a phase factor of absolute value one. So, for the second factor on the right hand side of
\eqref{holder} we use the endpoint Strichartz
inequality for the one-dimensional Schr\"odinger equation, thus producing $|k_1|^{\frac{1}{4}}\|\F_xu_0(k_1,\cdot)\|_{L^2_{y}}$,
where the $k_1$ factor comes from $dt = k_1 ds$ in $L^4_t$. By conservation of the
$L^2_y$ norm, the last factor is nothing but $\|\F_xv_0(k_2,\cdot)\|_{L^2_{y}}$. We thus get
\begin{gather*}
\left(\sum_{k_1>0}\left\|\psi \; e^{it(\phi_0(k_1)+\phi_0(k_2))}
e^{i \frac {t}{ k_1} \pd_y^2} \mathcal{F}_xu_0(k_1,\cdot) \;
e^{i \frac {t}{ k_2} \pd_y^2}\mathcal{F}_xv_0(k_2,\cdot) \right\|^2_{L^2_{ty}} \right)^{\frac 1 2} \\
\ls \left(\sum_{k_1>0}|k_1|^{\frac{1}{2}}\|\F_xu_0(k_1,\cdot)\|_{L^2_{y}}^2\|\F_xv_0(k_2,\cdot)\|_{L^2_{y}}^2 \right)^{\frac 1 2} 
\ls \|D_x^{\frac 1 4}u_0\|_{L^2_{xy}} \|v_0\|_{L^2_{xy}}.
\end{gather*}

Finally, when $|\frac{\tau-\tau_1}{k}| > 1 \Rightarrow |\tau-\tau_1|>k$, we exploit the use of the cutoff
function; the estimate
$$|\hat{\psi}(\tau-\tau_1)| \ls \frac{1}{\langle \tau-\tau_1 \rangle k^\beta},$$
is valid, for arbitrarily large $\beta$, because $\psi \in \mathcal{S}(\R)$ (with the inequality constant
depending only on $\psi$ and $\beta$). Fixing any such $\beta >1$, we can write
\begin{multline*}
\int\sum_{k_1>0} |\hat{\psi}(\tau-\tau_1)|\frac{k_1k_2}{k\:\omega(\tau_1)}
 \Big|\widehat{u_0}(k_1,\frac{\eta k_1}{k}\pm \omega(\tau_1))
 \widehat{v_0}(k_2,\frac{\eta k_2}{k}\mp \omega(\tau_1))\Big| 
\chi_{\{\omega(\tau_1)^2 \le k_1k_2\}}\chi_{\{|\frac{\tau-\tau_1}{k}| > 1\}}d\tau_1  \\
\ls  \int \frac{1}{\langle \tau-\tau_1 \rangle} \sum_{k_1>0} \frac{1}{(k_1k_2)^{\beta / 2}}\frac{k_1k_2}{k\:\omega(\tau_1)}
 \Big|\widehat{u_0}(k_1,\frac{\eta k_1}{k}\pm \omega(\tau_1))
 \widehat{v_0}(k_2,\frac{\eta k_2}{k}\mp \omega(\tau_1))\Big| 
\chi_{\{\omega(\tau_1)^2 \le k_1k_2\}} d\tau_1.
\end{multline*}
The $L^2_{\tau}$ norm of this quantity is bounded, using the same convolution estimate as before, by
\begin{multline*}
\|\langle \cdot \rangle^{-1} \|_{L^2_{\tau}}
\int \sum_{k_1>0} \frac{1}{(k_1k_2)^{\beta / 2}}\frac{k_1k_2}{k\:\omega(\tau)}
 \Big|\widehat{u_0}(k_1,\frac{\eta k_1}{k}\pm \omega(\tau))
 \widehat{v_0}(k_2,\frac{\eta k_2}{k}\mp \omega(\tau))\Big| 
\chi_{\{\omega(\tau)^2 \le k_1k_2\}} d\tau \\
\ls
\sum_{k_1>0} \frac{1}{(k_1k_2)^{\beta / 2}} \int_{\omega\le \sqrt{k_1k_2}}
 \Big|\widehat{u_0}(k_1,\frac{\eta k_1}{k}\pm \omega)
 \widehat{v_0}(k_2,\frac{\eta k_2}{k}\mp \omega)\Big|  d\omega,
\end{multline*}
where we have done again the change of variables of integration $d \tau =\frac{2\omega k}{k_1k_2}d\omega$. Applying 
H\"older's inequality to the integral, we then get
\begin{multline*}
\sum_{k_1>0} \frac{1}{(k_1k_2)^{\beta / 2}} |k_1 k_2|^{1/4} \bigg(\int
 \Big|\widehat{u_0}(k_1,\frac{\eta k_1}{k}\pm \omega)
 \widehat{v_0}(k_2,\frac{\eta k_2}{k}\mp \omega)\Big|^2  d\omega\bigg)^{1/2}\\
\ls \bigg(\sum_{k_1>0} \int
 \Big|\widehat{u_0}(k_1,\frac{\eta k_1}{k}\pm \omega)
 \widehat{v_0}(k_2,\frac{\eta k_2}{k}\mp \omega)\Big|^2  d\omega\bigg)^{1/2},
\end{multline*}
valid for our initial choice of $\beta$. The proof is complete, once we take the
$L^2_{k\eta}$ norm of this last formula, which is obviously bounded by $\|u_0\|_{L^2_{xy}}\|v_0\|_{L^2_{xy}}$. 
\end{proof}

\subsection{Proof of the Strichartz estimate in the $\T \times \R^2$ case}

\begin{proof}[Proof of Theorem \ref{secondStr}] 
We start by proving the easier case, when $s_{1,2} > 0$. Using again the Schr\"odinger point of view, 
as in the proof of Theorem \ref{firstStr}, the partial Fourier transform in the
$x$ variable yields
\[\F_x \big(e^{it\phi(D)}u_0\big) (k,y) = e^{it\phi_0(k)}e^{i\frac{t}{k}\Delta_{y}}\F_xu_0(k,y),\]
and hence
$$
\F_x e^{it\phi(D)}u_0\;e^{it\phi(D)}v_0 (k,y) 
= \sum_{\substack{k_1 \ne 0 \\ k_2=k-k_1 \ne 0}}e^{it\phi_0(k_1)}e^{it\phi_0(k_2)}
e^{i\frac{t}{k_1}\Delta_{y}}\F_xu_0(k_1,y)e^{i\frac{t}{k_2}\Delta_{y}}\F_xv_0(k_2,y).
$$
By Plancherel in the $x$ variable and Minkowski's inequality we see that
$$
\|e^{it\phi(D)}u_0\;e^{it\phi(D)}v_0\|_{L^2_{txy}}  
 \ls  \Big\|\sum_{\substack{k_1 \\ k_1,k_2 \ne 0}}\|e^{i\frac{t}{k_1}\Delta_{y}}\F_xu_0(k_1,\cdot)
e^{i\frac{t}{k_2}\Delta_{y}}\F_xv_0(k_2,\cdot)\|_{L^2_{ty}}\Big\|_{L^2_k}.
$$
H\"older's inequality and Strichartz's estimate for Schr\"odinger in two dimensions, with suitably chosen
admissible pairs, give
\begin{equation}
\label{nonendpoint}
\|e^{i\frac{t}{k_1}\Delta_{y}}\F_xu_0(k_1,\cdot)\;e^{i\frac{t}{k_2}\Delta_{y}}\F_xv_0(k_2,\cdot)\|_{L^2_{ty}} 
\ls |k_1|^{\frac{1}{p_1}}|k_2|^{\frac{1}{p_2}}\|\F_xu_0(k_1,\cdot)\|_{L^2_{y}}\|\F_xv_0(k_2,\cdot)\|_{L^2_{y}},
\end{equation}
where $\frac{1}{p_1}+\frac{1}{p_2}=\frac{1}{2}$ and $p_1,p_2 < \infty$\footnote{Because of the failure of the
endpoint Strichartz estimate in two dimensions, here we may \emph{not} admit $p_1 = \infty$ or $p_2 = \infty$.}.
Then, an easy convolution estimate in the $k_1$ variable yields
\begin{multline*}
\Big\|\sum_{\substack{k_1 \\ k_1,k_2 \ne 0}}
|k_1|^{\frac{1}{p_1}}\|\F_xu_0(k_1,\cdot)\|_{L^2_{y}}|k_2|^{\frac{1}{p_2}}\|\F_xv_0(k_2,\cdot)\|_{L^2_{y}}\Big\|_{L^2_k}\\
\ls \||k|^{\frac{1}{p_1}}\F_xu_0(k,\cdot)\|_{L^2_{k}L^2_{y}} \sum_{k \ne 0} |k|^{\frac{1}{p_2}}\|\F_xv_0(k,\cdot)\|_{L^2_{y}},
\end{multline*}
so that, Cauchy-Schwarz in $\sum_{k \ne 0}$ finally gives
$$
\|e^{it\phi(D)}u_0\;e^{it\phi(D)}v_0\|_{L^2_{txy}}\ls \|u_0\|_{H^{s_1}_x L^2_{y}}\|v_0\|_{H^{s_2}_xL^2_{y}},
$$
with $s_1=1/p_1$ and $s_2>1/p_2 + 1/2$.

For the case in which $s_1=0$ or $s_2=0$, we need to be able to replace \eqref{nonendpoint} by the endpoint inequality,
where all the derivatives fall on just one function
\begin{equation}\label{endpoint}
\|e^{i\frac{t}{k_1}\Delta_y}\F_xu_0(k_1,\cdot)e^{i\frac{t}{k_2}\Delta_y}\F_xv_0(k_2,\cdot)\|_{L^2_{yt}}
 \ls 
|k_1|^{\frac{1}{2}}\|\F_xu_0(k_1,\cdot)\|_{L^2_y}\|\F_xv_0(k_2,\cdot)\|_{L^2_y},
\end{equation}
from which the proof of \eqref{str3} for this case follows exactly as previously.

To establish \eqref{endpoint} we start by noting again, as in the previous section, that
it is enough to consider $k_1,k_2 >0$. We write $f(y)=\F_xu_0(k_1,y)$ and $g(y)=\F_xv_0(k_2,y)$. Then
\[\F_{ty}(e^{i\frac{t}{k_1}\Delta_y}f \; e^{i\frac{t}{k_2}\Delta_y}g)(\tau,\eta)
= \int_{\eta_2=\eta-\eta_1}  \delta \left(\tau - \frac{|\eta_1|^2}{k_1}- \frac{|\eta_2|^2}{k_2}\right)
\F_yf(\eta_1)\F_yg(\eta_2)d \eta_1.\]
Introducing $\omega:=\eta_1-\frac{k_1}{k}\eta$, so that $\eta_1=\frac{k_1}{k}\eta + \omega$, 
$\eta_2=\eta-\eta_1=\frac{k_2}{k}\eta - \omega$
and $k_2|\eta_1|^2+k_1|\eta_2|^2=k|\omega|^2+ \frac{k_1k_2}{k}|\eta|^2$, the latter becomes
\[\int  \delta(P(\omega))\;\F_yf \left(\frac{k_1}{k}\eta + \omega \right)
\F_yg \left(\frac{k_2}{k}\eta - \omega\right) d\omega,\]
where $P(\omega)=\tau- \frac{k}{k_1k_2}|\omega|^2-\frac{|\eta|^2}{k}$ 
with $|\nabla P(\omega)|=\frac{2k |\omega|}{k_1 k_2}$.
Using $\int  \delta(P(\omega))d\omega=\int_{P(\omega)=0}\frac{dS_{\omega}}{|\nabla P(\omega)|}$ 
and defining $r^2:= \frac{k_1k_2}{k}(\tau -\frac{|\eta|^2}{k})$, the previous integral can then be written as
\begin{multline*}
\frac{k_1k_2}{2kr}\int_{|\omega|=r}\F_yf \left(\frac{k_1}{k}\eta + \omega \right)
\F_yg \left(\frac{k_2}{k}\eta - \omega \right) dS_{\omega} \\
 \ls  \frac{k_1k_2}{k\sqrt{r}}\left(\int_{|\omega|=r} \left| \F_yf \left(\frac{k_1}{k}\eta + \omega \right)
\F_yg \left(\frac{k_2}{k}\eta - \omega\right)\right|^2dS_{\omega}\right)^\frac{1}{2},
\end{multline*}
by Cauchy-Schwarz with respect to the surface measure of the circle. By taking now the $L^2_{\tau}$ norm,
using $d\tau = 2\frac{k}{k_1k_2}rdr$, the result is
\begin{multline*}
\left(\frac{k_1k_2}{k}\int  \int_{|\omega|=r} \left|\F_yf \left(\frac{k_1}{k}\eta + \omega\right)
\:\F_yg \left(\frac{k_2}{k}\eta - \omega\right) \right|^2 dS_{\omega} \:dr \right)^\frac{1}{2}\\
= \left(\frac{k_1k_2}{k}\int  \left|\F_yf \left(\frac{k_1}{k}\eta + \omega\right)
\:\F_yg \left(\frac{k_2}{k}\eta - \omega \right) \right|^2 d\omega \right)^\frac{1}{2}.
\end{multline*}
It remains to take the $L^2_{\eta}$ norm. As above, we introduce new variables $\eta_+=\frac{\eta k_1}{k} + \omega$ and
$\eta_-=\frac{\eta k_2}{k} - \omega$, with Jacobian equal to one, yielding
\[\sqrt{\frac{k_1k_2}{k}}\|f\|_{L^2_y}\|g\|_{L^2_y}.\]
Since $k_2\le k$, by our sign assumption, the proof is complete.

\end{proof}

\emph{Remark:} We define the auxiliary norm
\[\|f\|_{\hat{L}^r_xL^p_tL^q_y}:= \|\F_xf\|_{L^{r'}_kL^p_tL^q_{y}},\]
where the $'$ denotes the conjugate H\"older exponent. Then a slight modification
of the above argument shows that
\begin{equation}\label{str3r}
\|e^{it\phi(D)}u_0\:e^{it\phi(D)}v_0\|_{\hat{L}^r_xL^2_{ty}} \ls \|u_0\|_{H^{s_1}_xL^2_{y}}\|v_0\|_{H^{s_2}_xL^2_{y}},
\end{equation}
provided $1 \le r \le 2$, $s_{1,2} > 0$ and $s_1 + s_2 > \frac12 + \frac1{r'}$.

\subsection{Counterexample for global Strichartz estimate in $\T \times \R$}

\begin{proof}[Proof of Proposition \ref{third}]
Let $\widehat{u_0}(\xi)=\widehat{v_0}(\xi)=\delta(k-N)\chi(\eta)$, where $N \gg 1$ and $\chi$ is the
characteristic function of an interval $I$, of length $2|I|$, symmetric around zero. In this case
$$I^{\pm}(\tau,\xi)=\delta(k-2N) \frac N 2 \frac{1}{\omega_N} \chi(\eta/2+\omega_N) \chi(\eta/2-\omega_N),$$
with 
$$\omega_N^2=N\phi_0(N)-\frac{N\tau}{2}-\frac{\eta^2}{4}.$$
By the support condition of $\chi$, we have
$$2|\omega_N| \le \big|\frac{\eta}{2}+\omega_N\big|+\big|\frac{\eta}{2}-\omega_N\big|\le 2|I|,$$
so that $\frac{1}{\omega_N}\ge \frac{1}{|I|}$.
Now,
\begin{eqnarray*}
\|I^{\pm}(\cdot,\xi)\|_{L^2_{\tau}}&=&\delta(k-2N)\frac N 2 \Big(\int \frac{1}{\omega_N^2}
 \chi(\eta/2+\omega_N) \chi(\eta/2-\omega_N) d\tau \Big)^{\frac 1 2}\\
&\cong &\delta(k-2N) N^{\frac 1 2} \Big(\int \frac{1}{\omega_N}
 \chi(\eta/2+\omega_N) \chi(\eta/2-\omega_N) d\omega_N \Big)^{\frac 1 2}\\
&\gs& \delta(k-2N) N^{\frac 1 2} |I|^{-\frac 1 2} |I|^{\frac 1 2}\chi(\eta)\\
&=& \delta(k-2N) N^{\frac 1 2}\chi(\eta),
\end{eqnarray*}
from which
$$\|I^{\pm}(\cdot,\xi)\|_{L^2_{\tau k \eta}}\sim N^{\frac 1 2}|I|^{\frac 1 2}.$$
On the other hand
$$\|D^s u_0\|_{L^2_{x y}}\|u_0\|_{L^2_{xy}}\sim N ^s |I|,$$
so that the estimate
$$\| \big(e^{it\phi(D)}u_0\big)^2\|_{L^2_{t x y}} \ls \|D_x^s u_0\|_{L^2_{x y}}\|u_0\|_{L^2_{x y}}$$
implies
$$N^{\frac 1 2 - s} \ls |I|^{\frac 1 2}.$$
Since we may have $|I|$ of any size we want, in particular $|I| \sim N^{\alpha}$, for any $\alpha \in 
\R$, we conclude that no $s \in \R$ would satisfy the condition.
\end{proof}

\section{Bilinear Estimates}
We start by recalling several
function spaces to be used in the sequel.  All these spaces are defined as the completion, with respect to the norms below,
of an appropriate space of smooth test functions $f$, periodic in the $x$- and rapidly
decreasing in the $y$- and $t$-variables, having the property
$\widehat{f}(\tau, 0,\eta)=0$. These norms depend on the
phase function $\phi(\xi)=\phi(k, \eta) = \phi_0(k)-\frac{|\eta|^2}{k}$, $\phi_0(k) =
|k|^{\alpha} k$, with $k \in \Z^*$ and $\eta \in \R$ or $\eta\in \R^2$ according to
whether we work in $\T\times \R$ or $\T\times\R^2$. We begin with the standard anisotropic
Bourgain norm
\begin{equation}\label{xsb1}
\|f\|_{X_{s_1,s_2,b}}:=\|\lb k\rb^{s_1}\lb\eta\rb^{s_2}\lb\tau-\phi(\xi)\rb^{b}\widehat{f}\|_{L^2_{\tau \xi}}.
\end{equation}
Also, for certain ranges of the dispersion exponent $\alpha$, we will have to use the spaces
$X_{s_1,s_2,b; \beta}$ with additional weights, introduced in \cite{B93} and defined by 
\begin{equation}\label{xsb2}
\|f\|_{X_{s_1,s_2,b;\beta}} := \left\|\langle k\rangle^{s_1}\lb\eta\rb^{s_2} \langle\tau-\phi(\xi)\rangle^b
\Big(1+\frac{\langle\tau-\phi(\xi)\rangle}{\langle k\rangle^{\alpha+1}}\Big)^{\beta}\hat f\right\|_{L^2_{\tau \xi}}.
\end{equation}

Recall that, for $b>1/2$, these spaces inject into the space of continuous flows on 
anisotropic Sobolev spaces $C(\R_t;H^{s_1,s_2})$, where naturally the Sobolev norms are given by
$$\|f\|_{H^{s_1,s_2}}:=\|\lb k\rb^{s_1}\lb\eta\rb^{s_2} \widehat{f}\|_{L^2_{\xi}}.$$

The classical KP-II equation, that is the case $\alpha=2$, becomes a limiting case in our considerations.
In this case, due to the periodicity in the $x$-variable, the parameter $b$ must necessarily have the
value $b=\frac12$. Consequently, in order to close the contraction 
mapping argument and to obtain the persistence property of the solutions, we shall use the 
auxiliary norms
\begin{equation}\label{ys}
\|f\|_{Y_{s_1,s_2;\beta}}:=\left\|\langle k\rangle^{s_1}\lb\eta\rb^{s_2} \langle\tau-\phi(\xi)\rangle^{-1}
\Big(1+\frac{\langle\tau-\phi(\xi)\rangle}{\langle k\rangle^{\alpha+1}}\Big)^{\beta}\hat f\right\|_{L^2_{\xi}(L^1_{\tau })},
\end{equation}
cf. \cite{GTV97}. Finally, we define
\begin{equation}\label{zs}
\|f\|_{Z_{s_1,s_2;\beta}}:=\|f\|_{Y_{s_1,s_2;\beta}}+\|f\|_{X_{s_1,s_2, -\frac12 ;\beta}}.
\end{equation}

Now, we state the bilinear estimates for the KP-II type equations on $\T\times\R$.
\begin{lemma}\label{b-est1}
Let $\alpha =2$. Then, for  $s_1 > - \frac14$ and $s_2 \ge 0$, there exist $\beta \in (0, \frac12)$ and $\gamma > 0$, 
such that, for all $u,v$ supported in $[-T,T] \times \T \times \R$,
\begin{equation}\label{bexy2}
\|\pd_x(uv)\|_{Z_{s_1,s_2;\beta}} \lesssim  T ^{\gamma}\|u\|_{X_{s_1,s_2, \frac12 ;\beta}} \|v\|_{X_{s_1,s_2, \frac12 ; \beta}}.
\end{equation}
\end{lemma}

\begin{lemma}\label{b-est2}
Let $2<\alpha \leq\frac52$. Then, for $s_1>\frac34-\frac{\alpha}2$ and $s_2 \ge 0$,
there exist $b'>-\frac12$ and $\beta \in [0, - b']$, such that,
 for all $b>\frac12$, 
\begin{equation}\label{bex.2}
\|\pd_x(uv)\|_{X_{s_1,s_2,b';\beta}}\lesssim \|u\|_{X_{s_1,s_2,b;\beta}}\|v\|_{X_{s_1,s_2,b;\beta}}.
\end{equation}
\end{lemma}

\emph{Remark:} While in the preceding two lemmas our estimates are at the line of optimality 
prescribed by the counterexample in the appendix, we lose optimality for higher dispersion. The reason
for this is that the low value of $s_1$, on the left hand side of the estimate, cannot be fully exploited
if the frequency $k$ of the product is very low 
compared with the frequencies $k_1$ and $k_2$ of each single factor. Especially, for the fifth order KP-II equation considered by Saut 
and Tzvetkov in \cite{ST99} and in \cite{ST00}, we cannot reach anything better than $s_1 > -\frac78$.

\begin{lemma}\label{b-est3}
Let $\alpha > \frac52$. Then, for $s_1>\frac18-\frac{\alpha}4$ and $s_2 \ge 0$,
there exists $b'>-\frac12$, such that,
 for all $b>\frac12$, the estimate \eqref{bex.2} 
holds true.
\end{lemma}

\quad

The bilinear estimates that we prove on $\T\times \R^2$ are:

\begin{lemma}\label{est0}
Let $\alpha =2$. Then, for $s_1 \ge \frac12$ and $s_2 > 0$, there exists $\gamma > 0$, 
such that, for all $u,v$ supported in $[-T,T] \times \T \times \R ^2$, the estimate
\begin{equation}\label{bexy3}
\|\pd_x(uv)\|_{Z_{s_1,s_2;\frac12}} \lesssim  T ^{\gamma}\|u\|_{X_{s_1,s_2, \frac12 ; \frac12}} \|v\|_{X_{s_1,s_2, \frac12 ; \frac12}},
\end{equation}
holds true.
\end{lemma}

\begin{lemma}\label{est1}
Let $2<\alpha \leq 3$. Then, for $s_1 > \frac{3-\alpha}{2}$ and $s_2\ge 0$, there exist 
$b'>-\frac12$ and $\beta \in [0, - b']$, such that, for all $b>\frac{1}{2}$,
\begin{equation}\label{nonlin-2}
\|\pd_x(uv)\|_{X_{s_1,s_2,b';\beta}}\ls \|u\|_{X_{s_1,s_2,b;\beta}}\|v\|_{X_{s_1,s_2,b;\beta}}.
\end{equation}
\end{lemma}

\begin{lemma}\label{est2}
Let $\alpha > 3$.  Then, for $s_1 > \max{(\frac{3-\alpha}{2}, \frac{1-\alpha}{4})}$ and $s_2\ge 0$, there exists 
$b'>-\frac{1}{2}$, such that, for all $b>\frac{1}{2}$, 
\begin{equation}\label{nonlin}
\|\pd_x(uv)\|_{X_{s_1,s_2,b'}}\ls \|u\|_{X_{s_1,s_2,b}}\|v\|_{X_{s_1,s_2,b}}.
\end{equation}
\end{lemma}

Before providing proofs of these lemmas, let us record some observations regarding the
norms 
to be used and the resonance relation associated to the KP-II type equations.

First of all, note that, for $s_2 \ge 0$, the following inequality holds:
$$ \frac{\lb \eta \rb^{s_2}}{\lb \eta_1 \rb^{s_2}\lb \eta_2 \rb^{s_2}} \lesssim 1,$$
which, applied to the inequalities \eqref{bexy2},\eqref{bex.2}, \eqref{nonlin-2} and
\eqref{nonlin}, allows us to reduce their proofs to the case $s_2=0$. Therefore,
for simplicity, throughout the remaining part of this paper, we abbreviate $X_{s,b}:=X_{s,0,b}$ and 
$X_{s,b;\beta}:=X_{s,0,b;\beta}$. We do the same for the anisotropic Sobolev spaces $H^s:=H^{s,0}$ 
as well as for the spaces $Y_{s;\beta}:=Y_{s,0;\beta}$ and $Z_{s;\beta}:=Z_{s,0;\beta}$
\footnote{To avoid 
confusion, we always put a semicolon in front of the exponent of the additional weights. If there is no 
semicolon, this exponent is zero.}. Only in the case $\alpha = 2$ of three space dimensions, where we have 
to admit an $\varepsilon$ derivative loss on the $y$-variable, shall we really need all the four parameters.

We write the $X_{s,b}$ norm in the following way
$$\|f\|_{X_{s,b}} = \|D^s_x\Lambda^b f\|_{L^2_{txy}},$$
where $D_x^s$ and $\Lambda^b$ are defined via the Fourier transform by $D_x^s= \F^{-1}\la k\ra^s\F$ and $\Lambda^b = \F^{-1}\la \tau-\phi(k, \eta)\ra^b\F$,
respectively. In the proof of Lemma \ref{est0} we will use $D_y^s= \F^{-1}\la \eta \ra^s\F$, too. Let us also 
introduce the notations $\sigma:= \tau -\phi(k, \eta)$, $\sigma_1 := \tau_1 -\phi(k_1, \eta_1)$ 
and $\sigma_2 := \tau-\tau_1 -\phi(k-k_1, \eta-\eta_1)$.
 For $\phi_0(k) = |k|^{\alpha}k$,  $\alpha > 0$, from  \cite{H}, we have that 
$$r(k, k_1) = \phi_0(k) -\phi_0(k_1) -\phi_0(k-k_1),$$
satisfies
\begin{equation}\label{reso-1}
 \frac{\alpha}{2^{\alpha}}|k_{min}| |k_{max}|^{\alpha}\leq |r(k,k_1)|
\leq (\alpha+1+\frac1{2^{\alpha}}) |k_{min}||k_{max}|^{\alpha}.
\end{equation}

We have  the resonance relation
\begin{equation}\label{reso-2}
 \sigma_1+\sigma_2-\sigma = r(k,k_1) + \frac{|k\eta_1-k_1\eta|^2}{kk_1(k-k_1)}.
\end{equation}

Note that both terms on the right hand side of \eqref{reso-2} have the same sign, 
so we have $|\sigma_1+\sigma_2-\sigma|\geq |r(k, k_1)|$. Therefore, from \eqref{reso-1} and \eqref{reso-2} we get the 
following lower bound for the resonance
\begin{equation}\label{reso-3}
 \max\{|\sigma|, |\sigma_1|, |\sigma_2|\} 
\geq \frac{\alpha}{3\;2^{\alpha}}|k_{min}| |k_{max}|^{\alpha}.
\end{equation}
In what follows, the lower bound \eqref{reso-3} plays an important role in the proof of the bilinear estimates.
 
While we have stated our central estimates in the canonical order, we will start with the proof of the simplest 
case and then proceed to the more complicated ones, partly referring to arguments used before. That's why we 
begin with three space dimensions.

\subsection{Proof of the bilinear estimates in the $\T\times \R ^2$ case}

Besides the resonance relation \eqref{reso-3} the following $X_{s,b}$-version of the bilinear
Strichartz estimate will be the key ingredient in our proofs in this section: combining \eqref{str3r}
with (a straightforward bilinear generalization of) Lemma 2.3 from \cite{GTV97}, we obtain

\begin{equation}\label{xstr}
\|uv\|_{\hat{L}^r_xL^2_{ty}} \ls \|u\|_{X_{s_1, b}} \|v\|_{X_{s_2,b}},
\end{equation}
and, by duality,
\begin{equation}\label{xstr'}
\|uv\|_{X_{-s_1, -b}}\ls \|u\|_{\hat{L}^{r'}_xL^2_{ty}}\|v\|_{X_{s_2,b}},
\end{equation}
provided $1 \le r \le 2$, $b> \frac12$, $s_{1,2} > 0$ and $s_1 + s_2 > \frac12 + \frac1{r'}$.
(For $r=2$ we can even admit $s_1 = 0$ or $s_2=0$ here.)
Taking $r=2$ in both estimates above we may interpolate between them, which gives
\begin{equation}\label{thetastr}
\|uv\|_{X_{-s_0, -b_0}}\ls \|u\|_{X_{s_1, b_1}} \|v\|_{X_{s_2,b}},
\end{equation}
whenever the parameters appearing are nonnegative and fulfill the conditions $s_0+s_1+s_2>1$, $b_0+b_1>\frac12$
as well as $b_1s_0=s_1b_0$.

\begin{proof}[Proof of Lemma \ref{est2}] We divide the proof in different cases.
 In all these cases we choose
$b'$ close to $-\frac12$ so that $b' \le - \frac1{\alpha}$, $s> 2 + (\alpha +1)b'$ and $s > \frac14 + \frac{\alpha b'}2$. Then we can find an
auxiliary parameter $\delta \ge 0$ (which may differ from case to case) such that the conditions
\begin{equation}\label{c1}
1+ \alpha b' + \delta \le 0 \quad \mbox{and} \quad b' + 1 - \delta < s,
\end{equation}
or
\begin{equation}\label{c2}
 \alpha b' + \delta \le s \quad \mbox{and} \quad b'+2 - \delta < 0,
\end{equation}
are fulfilled.

{\bf Case a:} Here we consider $\lb \sigma \rb \ge \lb \sigma_{1,2}\rb$. By symmetry we may assume $|k_1|\ge |k_2|$.

{\bf Subcase a.a:} $|k_2|\ls |k|$. Here we use the resonance relation \eqref{reso-3}, the bilinear estimate \eqref{xstr} and the condition
\eqref{c1} to obtain
\begin{equation*}
\begin{split}
\|D_x^{s+1}(uv)\|_{X_{0,b'}}&\lesssim \|(D_x^{s+1+\alpha b'+ \delta} u)(D_x^{b'-\delta}v)\|_{L^2_{txy}}\\
&\lesssim \|D_x^{s+1+\alpha b'+ \delta}u\|_{X_{0,b}} \|D_x^{(b'-\delta+1)+}v\|_{X_{0,b}} \le \|u\|_{X_{s,b}} \|v\|_{X_{s,b}}.
\end{split}
\end{equation*}

{\bf Subcase a.b:} If $|k| \ll |k_2|$, the resonance relation \eqref{reso-3} gives
\[\|D_x^{s+1}(uv)\|_{X_{0,b'}} \lesssim \|D_x^{s+1+b'}((D_x^{\alpha b'+ \delta} u)(D_x^{-\delta}v))\|_{L^2_{txy}},\]
which can be estimated as before as long as $s+1+b' \ge 0$. If $s+1+b' \in [- \frac12 , 0)$, we choose $\frac1{r'}=s+ \frac32+b'+$
and use a Sobolev type embedding, as well as \eqref{xstr}, to estimate the latter by
\[\|D_x^{\alpha b'+ \delta}u\|_{X_{0,b}} \|D_x^{(s+2+b'-\delta)+}v\|_{X_{0,b}} \le \|u\|_{X_{s,b}} \|v\|_{X_{s,b}}\]
where the last inequality follows from \eqref{c2}. If $s+1+b'< - \frac12$, we use a Sobolev type embedding and \eqref{xstr} to obtain the bound
\[\|(D_x^{\alpha b'+ \delta}u)(D_x^{-\delta}v)\|_{\hat{L}^1_xL^2_{ty}} \lesssim 
 \|D_x^{\alpha b'+ \delta}u\|_{X_{0,b}} \|D_x^{(\frac12-\delta)+}v\|_{X_{0,b}} \le \|u\|_{X_{s,b}} \|v\|_{X_{s,b}},\]
since $s > \frac14 + \frac{\alpha b'}{2}$.

{\bf Case b:} Next we consider $\sigma_1$ maximal. We further divide this case into three subcases.

{\bf Subcase b.a:} $|k|,|k_1|\ge|k_2|$. Using \eqref{reso-3}, the contribution from this subcase is bounded by
\begin{equation}\label{ba}
\begin{split}
&\|(D_x^{\alpha b'+1+\delta +s}\Lambda^b u)(D_x^{b'-\delta}v)\|_{X_{0,-b}} \\
\ls  & \|D_x^{\alpha b'+1+\delta +s}\Lambda^b u\|_{L^2_{txy}} \|D_x^{(1+b'-\delta)+}v\|_{X_{0,b}} \le \|u\|_{X_{s,b}} \|v\|_{X_{s,b}}, 
\end{split}
\end{equation}
where \eqref{xstr'} and \eqref{c1} are used here.

{\bf Subcase b.b:} $|k_{1,2}|\ge |k|$. Here we get the bound
\begin{equation}\label{bb}
\|D_x^{s+1+b'}(D_x^{\alpha b'+ \delta}\Lambda^b u \cdot D_x^{-\delta}v)\|_{X_{0,-b}},
\end{equation}
which is controlled by \eqref{ba} as long as $s+1+b' \ge 0$. If $s+1+b' \in [- 1 , 0)$, we use \eqref{xstr'} with $-s_1=s+1+b'$ and
the condition \eqref{c2}
to obtain the upper bound
\[\|D_x^{\alpha b'+ \delta} u\|_{X_{0,b}}\|D_x^{(s+2+b'-\delta)+}v\|_{X_{0,b}}\le \|u\|_{X_{s,b}} \|v\|_{X_{s,b}}.\]
If $s+1+b' < -1$, the same argument gives (with a certain waste of derivatives) the upper bound
\[\|D_x^{\alpha b'+ \delta} u\|_{X_{0,b}}\|D_x^{-\delta}v\|_{X_{0,b}}\le \|u\|_{X_{s,b}} \|v\|_{X_{s,b}},\]
as long as $s>\frac{\alpha b'}{2}$, which is a weaker demand as in subcase a.b.

{\bf Subcase b.c:} $|k|,|k_2|\ge|k_1|$. Here we use \eqref{reso-3} and \eqref{xstr'} with $r=1$ and a Sobolev type embedding
to obtain
\begin{equation}\label{bc}
\begin{split}
\|D_x^{s+1}(uv)\|_{X_{0,b'}}& \ls \|(D_x^{b'-\delta}\Lambda^b u)(D_x^{s+1+\alpha b'+ \delta}v)\|_{X_{0,-b}} \\
& \ls \|D_x^{b'-\delta}\Lambda^b u\|_{\hat{L}_x^{\infty}L^2_{ty}}\|D_x^{(s+\frac32 + \alpha b'+\delta)+}v\|_{X_{0,b}}\\
& \ls \|D_x^{b'-\delta + \frac12 +}u\|_{X_{0,b}}\|D_x^{(s+\frac32 + \alpha b'+\delta)+}v\|_{X_{0,b}}.
\end{split}
\end{equation}
Since $s> 2 + (\alpha +1)b'$ and $\alpha > 3$ we can choose $\delta \ge 0$ with $b'-\delta + \frac12<s$ and $\frac32 + \alpha b'+\delta<0$,
so that the latter is bounded by $c \|u\|_{X_{s,b}} \|v\|_{X_{s,b}}$.
\end{proof}

\emph{Remark:} Observe that the assumption $\alpha > 3$ is only needed in subcase b.c. In all the other subcases the arguments
presented work also for $2 < \alpha \le 3$, and the only relevant lower bound on $s$ in this range of $\alpha$ is $s>\frac{3-\alpha}{2}$.

\begin{proof}[Proof of Lemma \ref{est1}] 
Here we assume without loss that $s \le \frac12$ and choose $b'$ close to $-\frac12$, so that
$s> 2 + (\alpha +1)b'$ and that $\beta:=\frac{s+1+b'}{\alpha}\in [0,-b']$. Concerning the spaces
$X_{s,b;\beta}$ we recall that for $\beta \ge 0$ we have
\begin{equation}\label{5.1}
\|f\|_{X_{s,b}}\leq \|f\|_{X_{s,b;\beta}}.
\end{equation}
and that
\begin{equation}\label{5.2}
\|f\|_{X_{s,b}}\sim \|f\|_{X_{s,b;\beta}},
\end{equation}
if $\lb \sigma \rb \le \langle k\rangle^{\alpha+1}$. First we consider

{\bf Case a:} $\lb \sigma \rb \ge \langle k\rangle^{\alpha+1}$. In this case we have
\begin{equation}\label{5.3}
\|D_x^{s+1}(uv)\|_{X_{0,b';\beta}} \sim \|D_x^{s+1-\beta(\alpha +1)}(uv)\|_{X_{0,b'+\beta}}.
\end{equation}
We divide this case into two further subcases.

{\bf Subcase a.a:} $|k_{1,2}|\ls |k|$. By symmetry we may assume that $|k_1|\ge |k_2|$, then
\eqref{5.3} is bounded by
\begin{equation*}
\begin{split}
\|(D_x^{s+1+b'(\alpha +1)}u)v\|_{L^2_{txy}} &\ls \|D_x^{(2+b'(\alpha +1))+}u\|_{X_{0,b}}\|v\|_{X_{s,b}}\\
\ls \|u\|_{X_{s,b}} \|v\|_{X_{s,b}} & \ls \|u\|_{X_{s,b;\beta}} \|v\|_{X_{s,b;\beta}},
\end{split}
\end{equation*}
where we have used \eqref{xstr} with $s_2=s$, the assumption $s> 2 + (\alpha +1)b'$ and \eqref{5.1}.

{\bf Subcase a.b:} $|k| \ll |k_1| \sim |k_2| $. First assume that $\sigma$ is maximal. With this
assumption we get from \eqref{reso-3} that \eqref{5.3} is dominated by
\begin{equation*}
\begin{split}
\|D_x^{s+1+b'-\alpha \beta}(D_x^{\frac{b'+\beta}2 \alpha}u \cdot D_x^{\frac{b'+\beta}2 \alpha}v)\|_{L^2_{txy}}
 & =\|D_x^{\frac{b'+\beta}2 \alpha}u \cdot D_x^{\frac{b'+\beta}2 \alpha}v\|_{L^2_{txy}} \\
\ls \|D_x^{\frac{b'+\beta}2 \alpha + \frac12 +}u\|_{X_{0,b}}\|D_x^{\frac{b'+\beta}2 \alpha + \frac12 +}v\|_{X_{0,b}}&\le \|u\|_{X_{s,b}} \|v\|_{X_{s,b}},
\end{split}
\end{equation*}
by our choice of $\beta$, \eqref{xstr} with $s_1=s_2=\frac12+$, and the fact that
$s > \frac{b'+\beta}2 \alpha + \frac12$, which is a consequence of our choice of $\beta$
and $s> 2 + (\alpha +1)b'$.

If $\sigma_1$ is maximal, we obtain similarly as upper bound for \eqref{5.3}
\begin{equation*}
\begin{split}
\|D_x^{\frac{b'+\beta}2 \alpha} \Lambda ^bu \cdot D_x^{\frac{b'+\beta}2 \alpha}v\|_{X_{0,-b}} & \sim 
\|D_x^{\frac{b'+\beta}2 \alpha + \frac12} \Lambda ^bu \cdot D_x^{\frac{b'+\beta}2 \alpha - \frac12}v\|_{X_{0,-b}} \\
\ls \|D_x^{\frac{b'+\beta}2 \alpha + \frac12}u\|_{X_{0,b}}\|D_x^{\frac{b'+\beta}2 \alpha + \frac12 +}v\|_{X_{0,b}}
& \le \|u\|_{X_{s,b}} \|v\|_{X_{s,b}},
\end{split}
\end{equation*}
where we have used $|k_1| \sim |k_2|$, \eqref{xstr'}, and $s > \frac{b'+\beta}2 \alpha + \frac12$.

{\bf Case b:} $\lb \sigma \rb \le \langle k\rangle^{\alpha+1}$. In view of \eqref{5.2} we have to show that
\begin{equation}\label{5.4}
\|D_x^{s+1}(uv)\|_{X_{0,b'}}\ls \|u\|_{X_{s,b;\beta}} \|v\|_{X_{s,b;\beta}}.
\end{equation}
By earlier estimates - see the discussion of the subcases a.a, a.b, b.a, and b.b in the proof of
Lemma \ref{est2} - this has only to be done in the case where $\sigma_1$ is maximal and
$|k_1| \ll |k| \sim |k_2|$. Under these assumptions the additional weight in $\|u\|_{X_{s,b;\beta}}$
behaves like $\left(\frac{|k|}{|k_1|} \right)^{\alpha \beta}$, so that \eqref{5.4} reduces to
\begin{equation}\label{5.5}
\|D_x^{s+1-\alpha \beta}(uv)\|_{X_{0,b'}}\ls \|u\|_{X_{s-\alpha \beta,b}} \|v\|_{X_{s,b}}.
\end{equation}
Using again the resonance relation \eqref{reso-3} we estimate the left hand side of \eqref{5.5} by
\begin{equation}\label{5.6}
\begin{split}
 &\|D_x^{s+1-\alpha \beta}(D_x^{b'}\Lambda ^bu \cdot D_x^{\alpha b'}v)\|_{X_{0,-b}} \\
 \sim & \|D_x^{b'-\delta}\Lambda ^bu\cdot D_x^{s+1+\alpha (b' - \beta)+\delta }v \|_{X_{0,-b}} \\
 \ls & \|D_x^{b'-\delta}u\|_{X_{0,b}}\|D_x^{s+2+\alpha (b' - \beta)+\delta+ }v\|_{X_{0,b}},
\end{split}
\end{equation}
having used \eqref{xstr'} in the last step. Choosing $\delta=1+2b' >0$ the first factor becomes $\|u\|_{X_{s-\alpha \beta,b}}$,
and the number of derivatives in the second factor is $(2 + (\alpha +1)b')+ \le s$. Thus \eqref{5.5}
is shown and the proof is complete.
\end{proof}

To prove Lemma \ref{est0} we need a variant of \eqref{xstr} with $b < \frac12$. To obtain this, we first
observe that, if $s_{1,2}\ge 0$ with $s_1+s_2> \frac12$, $\e_{0,1,2} \ge 0$ with $\e_0 +\e_1+\e_2 >1$,
$1 \le p \le \infty$, and $b > \frac1{2p}$, then
\begin{equation}\label{x}
\|\F D_y^{-\e_0}(uv)\|_{L^2_{\xi}L^p_{\tau}} \ls \|u\|_{X_{s_1,\e_1,b}}\|v\|_{X_{s_2,\e_2,b}}.
\end{equation}
This follows from Sobolev type embeddings and applications of Young's inequality. Now bilinear
interpolation with the $r=2$ case of \eqref{xstr} gives the following.
\begin{kor}\label{varstr}
Let $s_{1,2}\ge 0$ with $s_1+s_2=1$ and $\e_{0,1,2} \ge 0$ with $\e_0 +\e_1+\e_2 >0$, then there
exist $b < \frac12$ and $p<2$ such that
\begin{equation}\label{xx}
\| D_y^{-\e_0}(uv)\|_{L^2_{txy}} \ls \|u\|_{X_{s_1,\e_1,b}}\|v\|_{X_{s_2,\e_2,b}},
\end{equation}
and \eqref{x} hold true.
\end{kor}
The purpose of the $p<2$ part in the above Corollary is to deal with the $Y$ contribution to the $Z$ norm
in Lemma \ref{est0}. Its application will usually follow on an embedding
$$\|\lb \sigma \rb^{-\frac12}\widehat{f}\|_{L^2_{\xi}L^1_{\tau}} \ls \|\widehat{f}\|_{L^2_{\xi}L^p_{\tau}} ,$$
where $p<2$ but arbitrarily close to $2$. We shall also rely on the dual version of \eqref{xx}, that is
\begin{equation}\label{xx'}
\|uv\|_{X_{-s_1,-\e_1,-b}}\ls \|\F D_y^{\e_0}u\|_{L^2_{\tau \xi}}\|v\|_{X_{s_2,\e_2,b}}.
\end{equation}

\begin{proof}[Proof of Lemma \ref{est0}] 

In this proof we will take $s_2 = \e$, $s_1 = s$ and restrict ourselves to the lowest value $s=\frac12$.
Again the proof consists of a case by case discussion.

{\bf Case a:} $\lb k\rb^3 \le \lb \sigma \rb$. First we observe that
\begin{equation}\label{+++}
\|\pd_x(uv)\|_{Z_{s,\e;\frac12}} \lesssim \|D_x^{s+1}(D_y^{\e}u\cdot v)\|_{Z_{0,0;\frac12}}+\|D_x^{s+1}(u\cdot D_y^{\e}v)\|_{Z_{0,0;\frac12}}.
\end{equation}
The first contribution to \eqref{+++} can be estimated by
\begin{equation*}
\begin{split}
\|\F(D_y^{\e}u\cdot v)\|_{L^2_{\tau \xi}} + \|\lb \sigma \rb^{-\frac12}\F(D_y^{\e}u\cdot v)\|_{L^2_{\xi}L^1_{\tau}} \\
\ls \|\F(D_y^{\e}u\cdot v)\|_{L^2_{\tau \xi} \cap L^2_{\xi}L^p_{\tau}} \ls \|u\|_{X_{s,\e,b}}\|v\|_{X_{s,\e,b}},
\end{split}
\end{equation*}
where we have used Corollary \ref{varstr}, for some $b < \frac12$. Using the fact\footnote{for a proof see e. g. Lemma 1.10 in \cite{AG}} that
under the support assumption on $u$ the inequality
\begin{equation}\label{time}
\|u\|_{X_{s,\e,b}}\ls T^{\tilde{b}-b}\|u\|_{X_{s,\e,\tilde{b}}},
\end{equation}
holds, whenever $-\frac12 < b < \tilde{b} < \frac12$, this can be further estimated by
$T^{\gamma}\|u\|_{X_{s,\e, \frac12 ; \frac12}} \|v\|_{X_{s,\e, \frac12 ; \frac12}}
$ for some $\gamma > 0$,
as desired. The second contribution to \eqref{+++} can be estimated in precisely the same manner.

{\bf Case b:} $\lb k\rb^3 \ge \lb \sigma \rb$. Here the additional weight on the left is of size one, so
that we have to show
$$\|\pd_x(uv)\|_{Z_{s,\e}} \lesssim T^{\gamma}\|u\|_{X_{s,\e, \frac12 ; \frac12}} \|v\|_{X_{s,\e, \frac12 ; \frac12}}
. $$

{\bf Subcase b.a:} $\sigma$ maximal. Exploiting the resonance relation \eqref{reso-3}, we see that the contribution
from this subcase is bounded by
$$\|\F D_xD_y^{\e}(D_x^{-\frac12}u\cdot D_x^{-\frac12}v)\|_{L^2_{\tau \xi} \cap L^2_{\xi}L^p_{\tau}} 
\ls \|\F(D_x^{\frac12}D_y^{\e}u\cdot D_x^{-\frac12}v)\|_{L^2_{\tau \xi} \cap L^2_{\xi}L^p_{\tau}} + \dots ,$$
where $p<2$. The dots stand for the other possible distributions of derivatives on the two factors,
in the same norms, which - by Corollary \ref{varstr} - can all be estimated by $c\|u\|_{X_{s,\e,b}}\|v\|_{X_{s,\e,b}}$
for some $b < \frac12$. The latter is then further treated as in case a.

{\bf Subcase b.b:} $\sigma_1$ maximal. Here we start with the observation that by Cauchy-Schwarz and \eqref{time},
for every $b' > -\frac12$ there is a $\gamma > 0$ such that
$$\|\pd_x(uv)\|_{Z_{s,\e}} \lesssim T^{\gamma}\|D_x^{s+1}(uv)\|_{X_{0,\e,b'}}.$$
Now the resonance relation gives
\begin{equation*}
\begin{split}
&\|D_x^{s+1}(uv)\|_{X_{0,\e,b'}} \ls \|D_x(D_x^{-\frac12}\Lambda^{\frac12}u \cdot D_x^{-\frac12}v)\|_{X_{0,\e,b'}}\\
\ls & \|(D_x^{\frac12}D_y^{\e}\Lambda^{\frac12}u)  (D_x^{-\frac12}v)\|_{X_{0,b'}}  + \|(D_x^{\frac12}\Lambda^{\frac12}u)  (D_x^{-\frac12}D_y^{\e}v)\|_{X_{0,b'}}\\
+ & \|(D_x^{-\frac12}D_y^{\e}\Lambda^{\frac12}u)  (D_x^{\frac12}v)\|_{X_{0,b'}}  + \|(D_x^{-\frac12}\Lambda^{\frac12}u)  (D_x^{\frac12}D_y^{\e}v)\|_{X_{0,b'}}.
\end{split}
\end{equation*}
Using \eqref{xx'} the first two contributions can be estimated by $c\|u\|_{X_{s,\e,\frac12}}\|v\|_{X_{s,\e,b}}$
as desired. The third and fourth term only appear in the frequency range $|k| \ll |k_1| \sim |k_2|$, where
the additional weight in the $\|u\|_{X_{s,\e,\frac12 ; \frac12}}$-norm on the right becomes $\frac{|k_2|}{|k_1|}$,
thus shifting a whole derivative from the high frequency factor $v$ to the low frequency factor $u$. So,
using \eqref{xx'} again, these contributions can be estimated by 
$$c\|u\|_{X_{s,\e,\frac12 ; \frac12}}\|v\|_{X_{s,\e,b}}\ls \|u\|_{X_{s,\e,\frac12 ; \frac12}}\|v\|_{X_{s,\e,b ; \frac12}}.$$
\end{proof}

\subsection{Proof of the bilinear estimates in the $\T\times \R$ case} In two space dimensions we have
the following $X_{s,b}$-version of Theorem \ref{firstStr}. Assume $s_{1,2}\ge 0$, $s_1+s_2=\frac14$ and
$b > \frac12$. Then, with a smooth time cut off function $\psi$,
\begin{equation}\label{xstr2}
\|\psi uv\|_{L^2_{txy}} \ls \|u\|_{X_{s_1, b}} \|v\|_{X_{s_2,b}}.
\end{equation}
The dual version of \eqref{xstr2} reads
\begin{equation}\label{xstr2'}
\|\psi uv\|_{X_{-s_1, -b}}\ls \|u\|_{L^2_{txy}}\|v\|_{X_{s_2,b}}.
\end{equation}
Until the end of this section we assume $u,\; v$ to be supported in $[-1, 1]\times \T\times\R$, so that we
can forget about $\psi$ in the estimates.

\quad

Let's revisit the proof of Lemma \ref{est2} in the previous section, replacing estimate \eqref{xstr} and its
dual version by the corresponding estimates \eqref{xstr2} and \eqref{xstr2'} valid in two dimensions, in
order to prove the pure (i. e. without additional weights) $X_{s,b}$-estimate 
$$\|\pd_x(uv)\|_{X_{s,b'}}\ls \|u\|_{X_{s,b}}\|v\|_{X_{s,b}},$$
where $b > \frac12$ and $s>\max{(\frac34-\frac{\alpha}2, }\frac18 - \frac{\alpha}{4})$. As above, we assume
$s \le 0$ and choose $b' > -\frac12$, but close to it, so that $b' < -\frac1{\alpha}$ (possible for $\alpha >2$)
and
\begin{equation}\label{s}
s > \frac54 + (\alpha +1)b', \qquad s>\frac18 + \frac{\alpha b'}{2}.
\end{equation}
Now we follow the case by case discussion from the proof of Lemma \ref{est2}.

The argument in subcase a.a works for all $\alpha >2$. Because there is only a loss of $\frac14$ derivative
in the application of \eqref{xstr2} (instead of $1+$, as in \eqref{xstr}), we are led to the condition
\begin{equation}\label{d1}
1+ \alpha b' + \delta \le 0 \quad \mbox{and} \quad b' + \frac14 - \delta < s,
\end{equation}
which replaces \eqref{c1} and can be fulfilled for some $\delta \ge 0$ because of our general assumption
\eqref{s}.

The argument in subcase a.b leads to the same condition, as long as $s+1+b'\ge 0$, i. e. for
$\alpha \le \frac52$. A possible Sobolev embedding does not give any improvement in the two-dimensional
setting. So, for $\alpha > \frac52$ this contribution is estimated roughly by
$$ \|(D_x^{\alpha b'+ \delta} u)(D_x^{-\delta}v)\|_{L^2_{txy}} \lesssim \|D_x^{\alpha b'+ \delta + \frac18}u\|_{X_{0,b}}\|D_x^{\frac18 - \delta} v\|_{X_{0,b}}\le \|u\|_{X_{s,b}} \|v\|_{X_{s,b}},$$
where we have used the second part of \eqref{s} in the last step.

In the discussion of subcase b.a we apply the dual version \eqref{xstr2'}, with $s_1=0$ instead of \eqref{xstr'},
and end up with condition \eqref{d1} again. The only restriction on $\alpha$ arising in this subcase is $\alpha >2$.

The estimate in subcase b.b is again reduced to that in subcase b.a, as long as $s+1+b' \ge 0$. For
$s+1+b' \in [- \frac14, 0]$, we use \eqref{xstr2'} with $-s_1=s+1+b'$. This leads to the condition
\begin{equation}\label{d2}
\alpha b' + \delta \le s \quad \mbox{and} \quad b' + \frac54 - \delta \le s,
\end{equation}
replacing \eqref{c2}, which again can be fulfilled choosing $\delta \ge 0$ appropriately by our general
assumption \eqref{s}. This works for $s+1+b' \ge - \frac14$, i. e. for $\alpha \le 3$. If $s+1+b' < - \frac14$ 
(corresponding to $\alpha > 3$) we use \eqref{xstr2'} with $s_1 = \frac14$ (thus wasting again several derivatives)
and end up with the condition $s > \frac{\alpha b'}{2}$, which is weaker than \eqref{s}.

Finally, we turn to subcase b.c ($\sigma_1$ maximal, $|k|,|k_2|\ge|k_1|$), where we used the resonance
relation \eqref{reso-3}, to obtain
$$\|D_x^{s+1}(uv)\|_{X_{0,b'}} \ls \|(D_x^{b'-\delta}\Lambda^b u)(D_x^{s+1+\alpha b'+ \delta}v)\|_{X_{0,-b}},$$
for some $\delta \ge 0$. Now we apply \eqref{xstr2'} to estimate the latter by
$$\|D_x^{b'-\delta}\Lambda^b u\|_{L^2_{txy}}\|D_x^{s+\frac54 + \alpha b'+\delta}v\|_{X_{0,b}}\ls \|u\|_{X_{s,b}} \|v\|_{X_{s,b}},$$
provided $b'-\delta\le s$ and $\frac54 + \alpha b'+\delta \le 0$. Summing up the last two conditions
we end up with our general assumption \eqref{s}, but for the second of them we need at least
$\frac54 + \alpha b' \le 0$, which requires $\alpha > \frac52$. Observe that in this case both
conditions can in fact be fulfilled for $b'$ close enough to \nolinebreak $- \frac12$.

\quad
Since for $\alpha > \frac52$ the condition $s > \frac18 - \frac{\alpha}4$ is stronger than $s>\frac34-\frac{\alpha}2$,
we have proven Lemma \ref{b-est3}.
Next we turn to the proof of Lemma \ref{b-est2}, which follows closely along the lines of that of Lemma \ref{est1}.

\begin{proof}[Proof of Lemma \ref{b-est2}]
With the assumptions on $s$ and $b'$, as in the preliminary consideration above, we choose $\beta:=\frac{s+1+b'}{\alpha}\in [0,-b']$.
We follow the case by case discussion in the proof of Lemma \ref{est1}, beginning with case a, where
$\lb \sigma \rb \ge \langle k\rangle^{\alpha+1}$, so that \eqref{5.3} holds. In subcase a.a, where
$|k_{1,2}|\ls |k|$, we merely replace the application of \eqref{xstr} by that of \eqref{xstr2}, which
is justified by assumption \eqref{s}. Similarly, in subcase a.b ($|k| \ll |k_1| \sim |k_2| $), under
the additional assumption that $\sigma$ is maximal, we use \eqref{xstr2} with $s_1=s_2=\frac18$ and
are led to the condition $2s \ge (b'+\beta)\alpha + \frac14$, which is a consequence of \eqref{s}.
The same condition arises, if, in this subcase, $\sigma_1$ is assumed to be maximal and the estimate
\eqref{xstr'} is replaced by \eqref{xstr2'}.

In case b, where $\lb \sigma \rb \le \langle k\rangle^{\alpha+1}$, we have to show \eqref{5.4}. By the
discussion preceding this proof, this needs to be done only for $\sigma_1$ being maximal and
$|k_1| \ll |k| \sim |k_2| $, which amounts to the proof of \eqref{5.4}. This works as in \eqref{5.6},
except for the last step, where we use \eqref{xstr2'} instead of \eqref{xstr'}. With the same choice
of $\delta$ the number of derivatives on the second factor becomes now $\frac54 + (\alpha + 1)b' \le s$,
by assumption \eqref{s}.
\end{proof}

Our next task is the proof of Lemma \ref{b-est1}, where a variant of \eqref{xstr2} with $b < \frac12$
is required. The latter will be obtained as before by interpolation with an auxiliary estimate, but
with the decisive difference that we have to avoid any derivative loss in the $y$ variable, in order to
obtain a local result in (and below) $L^2$ and hence something global by the conservation of the $L^2$-norm.
So the simple Sobolev embedding argument applied to obtain \eqref{x} is not sufficient in two space
dimensions. Instead of that we will prove the following Lemma, which is partly contained already in
\cite[Lemma 4]{ST01} as well as in the unpublished manuscript \cite{T&T} of Takaoka and Tzvetkov.

\begin{lemma}\label{aux}
For $s_0>\frac34$, $\frac12 \le \frac1p < \frac34$, and $b_0 > \frac58 - \frac1{2p}$ the following
estimate holds true:
$$\|\F ((D_x^{-s_0}u)v)\|_{L^2_{\xi}L^p_{\tau}} \ls \|\lb \sigma \rb^{b_0}\widehat{u}\|_{L^2_{\xi}L^p_{\tau}}\|\lb \sigma \rb^{b_0}\widehat{v}\|_{L^2_{\xi}L^p_{\tau}}.$$
\end{lemma}

\begin{proof}
Since $p$ is close enough to $2$, we may assume without loss that $b_0 < \frac1{p'}$. With
$f(\xi, \tau)=\lb \sigma\rb^{-b_0}\widehat{u}(\xi, \tau)$ and $g(\xi, \tau)=\lb \sigma\rb^{-b_0}\widehat{v}(\xi, \tau)$
we have
$$\F ((D_x^{-s_0}u)v)(\xi, \tau) = \int |k_1|^{-s_0}\frac{f(\xi_1, \tau_1)}{\lb \sigma_1\rb^{b_0}}\frac{g(\xi_2, \tau_2)}{\lb \sigma_2\rb^{b_0}}d\xi_1d\tau_1,$$
where $(\xi, \tau)=(k,\eta, \tau)=(k_1+k_2,\eta_1+\eta_2, \tau_1+\tau_2)=(\xi_1 + \xi_2, \tau_1+\tau_2)$, $\int d\xi_1d\tau_1= \sum_{k_1\neq 0 \neq k_2}\int d \eta_1d\tau_1$, and $\sigma_{1,2}=\tau_{1,2}- \phi(\xi_{1,2})$.
Concerning the frequencies $k$, $k_1$ and $k_2$ corresponding to the $x$-variable we will assume that
$0<|k_1|\le|k_2|\le|k|$, see again pg. 460 in \cite{ST01}. Applying H\"older's inequality with respect
to $\int d \tau_1$ and \cite[Lemma 4.2]{GTV97} we obtain
$$|\F ((D_x^{-s_0}u)v)(\xi, \tau)| \ls 
\int|k_1|^{-s_0}\left( \int|f(\xi_1, \tau_1)g(\xi_2, \tau_2)|^p d \tau_1\right) ^{\frac{1}{p}}
\lb\tau -\phi(\xi_{1})-\phi(\xi_{2})\rb^{\frac1{p'}-2b_0}d \xi_1.$$
We introduce new variables $\omega =\eta_1- \frac{k_1}{k}\eta$ and $\omega '= \frac{k}{k_1k_2}\omega ^2$,
write $|k_1|^{-s_0}=(|k_1|^{-s_1}|\omega '|^{-\e})(|k_1|^{-s_2}|\omega '|^{\e})$, where $s_0=s_1+s_2$,
$\e=\frac{s_1}{3}$ and apply H\"older's inequality with respect to $\int d \xi_1$ to obtain the upper bound
$$\dots \ls I(\xi, \tau)\left(\int|k_1|^{-s_1p}|\omega '|^{-\e p}|f(\xi_1, \tau_1)g(\xi_2, \tau_2)|^p d \xi_1d \tau_1\right) ^{\frac{1}{p}},$$
where, with $a= \tau -\phi_0(k_{1})-\phi_0(k_{2})+\frac{|\eta|^2}{k}$,
\begin{equation*}
\begin{split}
I(\xi, \tau)^{p'}&=\sum_{k_1\neq 0 \neq k_2}|k_1|^{-s_2 p'}\int|\omega '|^{\e p'} \lb a+ \omega '\rb^{1-2b_0p'} d \omega \\
&=c\sum_{k_1\neq 0 \neq k_2}|k_1|^{-s_2 p'+ \frac12}\int|\omega '|^{\e p' - \frac12} \lb a+ \omega '\rb^{1-2b_0p'} d \omega '.
\end{split}
\end{equation*}
The latter is bounded by a constant independent of $(\xi, \tau)$, provided
\begin{equation}\label{cond1}
\frac{s_1}3 \le \frac1{2p'}\quad;\quad 2b_0 - \frac{s_1}3 > \frac3{2p'}\quad;\quad s_2> \frac3{2p'}.
\end{equation}
The remaining factor can be rewritten and estimated by
$$\left( \int  |k_1(\eta_1-\frac{k_1}{k}\eta)|^{-2\e p}|f(\xi_1, \tau_1)g(\xi_2, \tau_2)|^p  d \xi_1 d \tau_1\right) ^{\frac{1}{p}}.$$
Taking the $L^2_{\xi}L^p_{\tau}$-norm of the latter, we arrive at
$$ \left\| \left( \int  |k_1(\eta_1-\frac{k_1}{k}\eta)|^{-2\e p}
\|f(\xi_1, \cdot)\|_{L^p_{\tau}}^p\|g(\xi_2, \cdot)\|_{L^p_{\tau}}^p d \xi_1\right)^{\frac{1}{p}} \right\|_{L^2_{\xi}}
\ls \|f\|_{L^2_{\xi}L^p_{\tau}}\|g\|_{L^2_{\xi}L^p_{\tau}},$$
where in the last step we have used H\"older's inequality (first in $\eta_1$, then in $k_1$), which requires
\begin{equation}\label{cond2}
s_1 > \frac3{2p}-\frac34.
\end{equation}
Finally our assumptions on $s_0$, $b_0$ and $p$ allow us to choose $s_1$ properly, so that the conditions
\eqref{cond1} and \eqref{cond2} are fulfilled.
\end{proof}

An application of H\"older's inequality in the $\tau$ variable gives:

\begin{kor}\label{xaux}
Let $s_0>\frac34$, $\frac12 \le \frac1p < \frac34$, and $b > \frac18 + \frac1{2p}$. Then the estimate
$$\|\F ((D_x^{-s_0}u)v)\|_{L^2_{\xi}L^p_{\tau}} \ls \|u\|_{X_{0,b}}\|v\|_{X_{0,b}},$$
is valid.
\end{kor}

Observe that the estimates in Lemma \ref{aux} and Corollary \ref{xaux} are valid without the general
support assumption on $u$ and $v$. This is no longer true for the next Corollary, which is obtained
via bilinear interpolation between \eqref{xstr2} and Corollary \ref{xaux}.

\begin{kor}\label{varstr2}
For $s_{1,2}\ge 0$, with $s_1+s_2>\frac14$, there exist $b < \frac12$ and $p<2$, such that
\begin{equation}\label{+}
\|uv\|_{L^2_{txy}} \ls \|u\|_{X_{s_1,b}}\|v\|_{X_{s_2,b}},
\end{equation}
and
\begin{equation}\label{++}
\|\F(uv)\|_{L^2_{\xi}L^p_{\tau}} \ls \|u\|_{X_{s_1,b}}\|v\|_{X_{s_2,b}}.
\end{equation}
\end{kor}

\begin{proof}[Sketch of proof of Lemma \ref{b-est1}]
To prove Lemma \ref{b-est1} we now insert Corollary \ref{varstr2} into the framework of the proof of
Lemma \ref{b-est2}. Assuming further on $s\le 0$, we especially take $\beta= \frac{s}2 + \frac14$,
which corresponds exactly to our choice in that proof. These arguments are combined with elements
of the proof of Lemma \ref{est0}. To extract a factor $T^{\gamma}$ we rely again on the estimate
\eqref{time}. The $p<2$ part of Corollary \ref{varstr2} serves to deal with the $Y$ contribution of
the $Z$ norm, whenever $\sigma$ is maximal. A corresponding argument can be avoided by a simple Cauchy-Schwarz
application in the case, where $\sigma_1$ is maximal. In this case we rely on the dual version of
\eqref{+}, that is
$$\|uv\|_{X_{-s_1,-b}}\ls \|u\|_{L^2_{txy}}\|v\|_{X_{s_2,b}},$$
with $s_{1,2}\ge 0$, $s_1+s_2>\frac14$ and $b < \frac12$. Further details are left to the reader.
\end{proof}

\section{Local Well-posedness}

To state and prove our local well-posedness results we use a cut-off function $\psi \in C_0^{\infty}$
with $0\leq \psi(t) \leq 1$ and
\begin{equation}
\psi(t) = \begin{cases} 1, \qquad |t|\leq 1\\ 0, \qquad |t|\geq 2.
					\end{cases}
\end{equation}
For $T>0$, we define $\psi_T(t) = \psi(\frac tT)$. Then our result concerning $\T \times \R$ reads as follows.

\begin{theorem}\label{loc.1}
Let $\alpha \ge 2$, $s_1 > \max{(\frac34-\frac{\alpha}2,\frac18-\frac{\alpha}4)}$ and $s_2 \ge 0$. Then,
for any $u_0\in H^{s_1,s_2}(\T\times \R)$ with zero $x$-mean, there exist  $b \ge \frac12$, $\beta \ge 0$, $T=T(\|u_0\|_{H^{s_1,s_2}})>0$ 
and a unique solution $u$ of the initial value problem \eqref{KPIId2}, defined on $[0,T] \times \T\times \R$ and satisfying
$\psi_T\;u\in X_{s_1,s_2,b;\beta}$. This solution is persistent and depends continuously on the initial data.
\end{theorem}

In three space dimensions, i. e. for data defined on $\T \times \R^2$, we have the following.

\begin{theorem}\label{loc.2}
Let $u_0\in H^{s_1,s_2}(\T\times \R^2)$ satisfy the mean zero condition. Then,
\begin{itemize}
\item[i.)] if $\alpha =2$, $s_1 \ge \frac12$ and $s_2 > 0$, there exist $T=T(\|u_0\|_{H^{s_1,s_2}})>0$ 
and a unique solution $u$ of \eqref{KPIId3} on $[0,T] \times \T\times \R^2$ satisfying
$\psi_T\;u\in X_{s_1,s_2,\frac12;\frac12}$,
\item[ii.)] if $\alpha >2$, $s_1 > \max{(\frac{3 - \alpha}2,\frac{1 - \alpha}4)}$ and $s_2 \ge 0$, there
exist $b > \frac12$, $\beta \ge 0$, $T=T(\|u_0\|_{H^{s_1,s_2}})>0$ and a unique solution $u$ of \eqref{KPIId3} 
on $[0,T] \times \T\times \R^2$ satisfying $\psi_T\;u\in X_{s_1,s_2,b;\beta}$.
\end{itemize}
In both cases the solutions are persistent and depend continuously on the initial data.
\end{theorem}

The proof of the above theorems follows standard arguments as can be found e. g. in \cite{B93}, \cite{GTV97}, or \cite{KPV96},
so we can restrict ourselves to several remarks. The key step is to apply the contraction mapping principle to the integral
equation corresponding to the initial value problems \eqref{KPIId2} and \eqref{KPIId3}, i. e.
\begin{equation}\label{duhamel}
 u(t) = e^{it\phi(D)}u_0 -\int_0^t e^{i(t-t')\phi(D)} uu_x(t')\,dt',
\end{equation}
more precisely, to its time localized version
\begin{equation}\label{duhamel-2}
 u(t) = \psi_1(t)e^{it\phi(D)} u_0 -\psi_T(t)\int_0^te^{i(t-t')\phi(D)}\psi_T(t')u(t')\psi_T(t')u_x(t')\,dt'=:\Phi(u(t)).
\end{equation}

Combining the linear estimates for $X_{s,b}$-spaces (see e. g. \cite[Lemma 2.1]{GTV97}), which are equally valid for the
spaces $X_{s_1,s_2,b;\beta}$, with the bilinear estimates from the previous section, one can check that the mapping $\Phi$
defined in \eqref{duhamel-2} is a contraction from a closed ball $\mathcal{B}_a \subset X_{s_1,s_2,b;\beta}$, of properly
chosen radius $a$, into itself. Here, a contraction factor $T^{\gamma}$, $\gamma > 0$, is obtained

\begin{itemize}
\item either from the linear estimate for the inhomogeneous equation, which works for $b> \frac12$, corresponding to $\alpha >2$,
\item or from the bilinear estimates as in Lemma \ref{b-est1} and in Lemma \ref{est0}, which is necessary in the limiting case, where 
$\alpha = 2$ and $b=-b'=\frac12$.
\end{itemize}

The persistence of the solutions obtained in this way follows from the embedding $X_{s_1,s_2,b;\beta} \subset C(\R, H^{s_1,s_2})$,
as long as $b > \frac12$, while for $b = \frac12$ this is a consequence of \cite[Lemma 2.2]{GTV97}. Concerning uniqueness
(in the whole space) and continuous dependence we refer the reader to the arguments in \cite[Proof of Theorem 1.5]{KPV96}.

\begin{appendix}

\section{Failure of regularity of the flow map in $\T \times \R$}
\label{illposedness}

We present in this appendix a type of ill-posedness result which shows that, 
in $\T \times \R$, our local well-posedness theorem of the previous section 
is optimal (except for the endpoint), 
as far as the use of the Picard iterative method based on the Duhamel formula goes. The result states that the data to solution map fails to 
be smooth at the origin, more specifically fails to be $C^3$, for the Sobolev regularities precisely below the 
range of the local existence theorem proved in the previous section, i.e. for $s<\frac 3 4 - \frac \alpha 2$. 
Because the Picard iteration method applied to the 
Duhamel formula yields, for small enough times, an analytic 
data to solution map, this lack of smoothness of the flow map excludes the possibility of proving local existence by this
scheme, at the corresponding lower regularity Sobolev spaces.

This proof is due to Takaoka and Tzvetkov, 
in an unpublished manuscript \cite{T&T} which, for completeness and due to its unavailability elsewhere 
in published form, 
is being reproduced here. It is done there for $\alpha =2$, which is the only case 
studied by the authors in that manuscript, but our adaptation
for any $\alpha \ge 2$ is obvious. Their proof is
inspired by the considerations of Bourgain in \cite{B97}, section 6,
where an analogous ill-posedness result is proved for the KdV equation, for $s<-3/4$, and it is equally similar
to N. Tzvetkov's own result, also for the KdV equation, in \cite{T99}.

\begin{theorem}

Let $s<\frac 3 4 - \frac \alpha 2$. There exists no $T>0$ such that \eqref{KPIId2} admits a unique local solution 
defined on $[-T,T]$, for which the 
data to solution map,
from $H^s(\T \times \R)$ 
to $H^s(\T \times \R)$ given by $u_0 \mapsto u(t)$, $t \in [-T,T]$,
is $C^3$ differentiable at zero.

\end{theorem}

\begin{proof}
Just as is done in \cite{B97} and \cite{T99}, consider, for $w \in H^s(\T \times \R)$ and $\delta \in \R$, 
the solution $u=u(\delta,t,x,y)$ to the 
Cauchy problem
\begin{equation}
\label{delta}
\begin{cases}
\pd_t u - |D_x|^{\alpha}\pd_x u+\pd_x^{-1}\pd_y^2 u +u\pd_x u=0, \\
u(\delta,0,x,y)=\delta w(x,y). 
\end{cases}
\end{equation}
Then, $u$ satisfies the integral equation
$$
u(\delta,0,x,y)=\delta e^{it\phi(D)}w - \int_0^t e^{i(t-t')\phi(D)}\:u\pd_x u\:dt'.  
$$
If, for a sufficiently small interval of time $[-T,T]$, the data to solution map of \eqref{delta} is of class $C^3$ at the origin, 
it yields a third order derivative
 $\frac{\pd^3 u}{\pd \delta^3}$, at 
$\delta=0$, with the property of being a bounded
multilinear operator from $(H^s(\T \times \R))^3$ to $H^s(\T \times \R)$, for any $t \in [-T,T]$. Explicit formulas
can be easily computed
$$\frac{\pd u}{\pd \delta}_{|\delta =0}= e^{it\phi(D)}w=
\sum_{k\ne0} \int_{-\infty}^{+\infty} e^{i(kx+\eta y)} e^{it(\phi_0(k)-\eta^2/k)} \hat{w}(k,\eta) d\eta,$$
\begin{eqnarray*}
\frac{\pd^2 u}{\pd \delta^2}_{|\delta =0}&= &
\int_0^t e^{i(t-t')\phi(D)} \pd_x \Big(\frac{\pd u}{\pd \delta}_{|\delta =0}\Big)^2dt' \\
&=&\int_{\R^2} \bigg\{ \sum_{\Gamma_1} e^{i\big(x(k_1+k_2) + y(\eta_1+\eta_2)\big)} \: e^{it\big(\phi_0(k_1+k_2)- \frac{(\eta_1+\eta_2)^2}{k_1+k_2}\big)}\\
&& \hspace{1.5cm} (k_1+k_2) \frac{e^{itA}-1}{A} \hat{w}(k_1,\eta_2)
\hat{w}(k_2,\eta_2)\bigg\} d\eta_1 d\eta_2,
\end{eqnarray*}
where $\Gamma_1=\{(k_1,k_2) \in \Z^2: k_1\ne 0, k_2\ne 0, k_1+k_2 \ne 0 \}$ and
\begin{eqnarray*}
A:=A(k_1,k_2,\eta_1,\eta_2)&=& \phi(\xi_1)+\phi(\xi_2)-\phi(\xi_1+\xi_2)\\
&&\hspace{-2cm}= \phi_0(k_1)+\phi_0(k_2)-\frac{\eta_1^2}{k_1}-\frac{\eta_2^2}{k_2}-\phi_0(k_1+k_2)
+\frac{(\eta_1+\eta_2)^2}{k_1+k_2}.
\end{eqnarray*}
Finally, the third derivative, at $\delta = 0$, is given by
\begin{multline*}
\frac{\pd^3 u}{\pd \delta^3}_{|\delta =0} = 
\int_0^t e^{i(t-t')\phi(D)} \pd_x \Big(\frac{\pd u}{\pd \delta}_{|\delta =0}\; \frac{\pd^2 u}{\pd \delta^2}_{|\delta =0}
\Big)dt' \\
=\int_{\R^3} \bigg\{ \sum_{\Gamma_2} \: e^{i\big(x(k_1+k_2+k_3)+y(\eta_1+\eta_2+\eta_3)\big)}\:e^{it\left(\phi_0(k_1+k_2+k_3)-
\frac{(\eta_1+\eta_2+\eta_3)^2}{k_1+k_2+k_3}\right)}\\
(k_1+k_2)(k_1+k_2+k_3) \frac{1}{A}\bigg[\frac{e^{it(A+B)}-1}{A+B}- \frac{e^{it B}-1}{B}\bigg]\\
\hat{w}(k_1,\eta_1)\hat{w}(k_2,\eta_2)\hat{w}(k_3,\eta_3) \bigg\} d\eta_1 d\eta_2 d\eta_3,
\end{multline*}
where $A$ is still defined as above, and now
$$
\Gamma_2=\{(k_1,k_2,k_3) \in \Z^3: k_j\ne 0, j=1,2,3,\; k_1+k_2 \ne 0, \;k_1+k_2+k_3 \ne 0 \},
$$
and
\begin{multline*}
B:=B(k_1,k_2,k_3,\eta_1,\eta_2,\eta_3)=\phi(\xi_3)+\phi(\xi_1+\xi_2)-\phi(\xi_1+\xi_2+\xi_3) \\
=\phi_0(k_3)-\frac{\eta_3^2}{k_3}+\phi_0(k_1+k_2)-\frac{(\eta_1+\eta_2)^2}{k_1+k_2}
-\phi_0(k_1+k_2+k_3)+ \frac{(\eta_1+\eta_2+\eta_3)^2}{k_1+k_2+k_3}.
\end{multline*}
It will be shown now that, for  $s<\frac 3 4 - \frac \alpha 2$, the necessary boundedness condition
\begin{equation}
\label{bound}
\left\|\frac{\pd^3 u}{\pd \delta^3}_{|\delta =0} \right\|_{H^s(\T \times \R)} \ls \|w\|^3_{H^s(\T \times \R)},
\end{equation}
fails for any $t\ne 0$, by using a carefully chosen function $w$. 

For that purpose, set
$$
w=w_N(x,y):= \sum_{\pm} \int_{-\beta N^{\frac 1 2}}^{\beta N^{\frac 1 2}}e^{\pm i N x}e^{ i \eta y} d\eta,
$$
where $\beta$ is to be chosen later, sufficiently small, and $N \gg 1$. Its Fourier transform is simply given by
$\widehat{w_N}(k,\eta)=\chi_{[-\beta N^{\frac 1 2}, \beta N^{\frac 1 2}]}(\eta)$ if $k=\pm N$, and zero otherwise.

To estimate $\left\|\frac{\pd^3 u}{\pd \delta^3}_{|\delta =0} \right\|_{H^s(\T \times \R)}$ from below
note that the main contribution to it comes from a combination of frequencies $(k_j,\eta_j)
\in \supp \, \widehat{w_N}$, $j=1,2,3$, such that the term $A+B$ is small (see \cite{B97} and \cite{T99} for very similar
reasoning). The $k$ frequencies necessarily always have to satisfy the relation $k_1=k_2=\pm N$, so that the
least absolute value for $A+B$ is achieved when $k_3$ has the opposite sign as $k_1$ and $k_2$, i.e. $k_3=\mp N$. 
In this situation, a cancellation of the expression
$$\phi_0(k_1)+\phi_0(k_2)+\phi_0(k_3)-\phi_0(k_1+k_2+k_3),$$
is obtained, so that we get
$$|A(k_1,k_2,\eta_1,\eta_2)+B(k_1,k_2,k_3,\eta_1,\eta_2,\eta_3)| \ls \beta,$$
and if $\beta$ is chosen very small,
$$\left|\frac{e^{it(A+B)}-1}{A+B} \right| \gs |t|.$$
Also
$$
|A(k_1,k_2,\eta_1,\eta_2)| \sim N^{\alpha + 1}.
$$
Therefore, one can derive the estimate
$$\left\|\frac{\pd^3 u}{\pd \delta^3}_{|\delta =0} \right\|_{H^s(\T \times \R)} \gs |t| N^s N^{-(\alpha + 1)} N^2 
N^{\frac 5 4}=|t| N^{s-\alpha+\frac 9 4},$$
whereas, clearly $\|w_N\|_{H^s(\T \times \R)} \ls N^{s+\frac 1 4}$.

We thus conclude that, for $t \ne 0$, \eqref{bound} fails for $s<\frac 3 4 - \frac \alpha 2$.
\end{proof}

A direct proof of the impossibility of determining a space $X_T$, continuously embedded in
$C([-T,T],H^s(\T \times \R))$, where the required estimates to perform a Picard iteration on 
the Duhamel formula hold, is given below.

\begin{theorem}

Let $s<\frac 3 4 - \frac \alpha 2$. There exists no $T>0$ and a space $X_T$, continuously embedded in 
$C([-T,T],H^s(\T \times \R))$, such that the following inequalities hold
\begin{equation}
\label{one}
\|e^{it\phi(D)}u_0\|_{X_T} \ls \|u_0\|_{H^s(\T \times \R)}, \qquad u_0 \in H^s(\T \times \R), 
\end{equation}
and
\begin{equation}
\label{two}
\left\|\int_0^t e^{i(t-t')\phi(D)}\:\pd_x (uv)\:dt'\right\|_{X_T} \ls \|u\|_{X_T}\|v\|_{X_T}, \qquad u,v \in X_T.
\end{equation}
Thus, it is not possible to apply the Picard iteration method, implemented on the Duhamel integral formula, for
any such space $X_T$.
\end{theorem}

\begin{proof}
If there existed a space $X_T$ such that \eqref{one} and \eqref{two} were true, then
\begin{multline*}
\left\|\int_0^t e^{i(t-t')\phi(D)}\:\pd_x\Big[e^{it'\phi(D)} u_0 \int_0^{t'}e^{i(t'-s)\phi(D)} 
\pd_x (e^{is\phi(D)} u_0)^2\:ds\Big]
\:dt'\right\|_{X_T} \ls \\
 \ls  \|e^{it\phi(D)} u_0\|^3_{X_T}
 \ls  \|u_0\|^3_{H^s(\T \times \R)}.
\end{multline*}
On the other hand, 
because $X_T$ is continuously embedded in $C([-T,T],H^s(\T \times \R))$ we would also
have 
$$\sup_{t \in [-T,T]}\|\cdot\|_{H^s(\T \times \R)} \ls \|\cdot\|_{X_T},$$
from which we would conclude that, for any $t \in [-T,T]$, and any $u_0 \in H^s(\T \times \R)$ 
the following inequality would hold:
\begin{multline*}
\left\|\int_0^t e^{i(t-t')\phi(D)}\:\pd_x\Big[e^{it'\phi(D)} u_0 \int_0^{t'}e^{i(t'-s)\phi(D)} 
\pd_x (e^{is\phi(D)} u_0)^2\:ds\Big]
\:dt'\right\|_{H^s(\T \times \R)} \\
\ls \|u_0\|^3_{H^s(\T \times \R)}.
\end{multline*}
But choosing $u_0$ as the function $w$ of the previous proof, we know that this estimate cannot hold
true if $s <\frac 3 4 - \frac \alpha 2$. 
\end{proof}

\end{appendix}

\noindent{\bf Acknowledgment:} The authors would like to thank J. C. Saut for having suggested our working on this
challenging problem. We benefited a lot by the unpublished manuscript \cite{T&T} of H. Takaoka
and N. Tzvetkov, to whom we are indebted. We specially wish to thank N. Tzvetkov for kindly providing a
copy of this manuscript. We are also grateful to S. Herr for pointing out to us the three-dimensional results
in \cite{HT}. Finally, the first author wishes to thank the Center of Mathematical Analysis, Geometry
and Dynamical Systems at the IST, Lisbon, for its kind hospitality during his visit, where this work was started.

\end{document}